\documentclass[11pt]{article}
\usepackage[utf8]{inputenc}
\usepackage[a4paper,top=2.5cm,bottom=2cm,left=2cm,right=2cm,marginparwidth=1.5cm]{geometry}
\usepackage{amsmath}
\usepackage{amssymb}
\numberwithin{equation}{section}
\usepackage[english]{babel}
\usepackage{amsthm}
\usepackage{bbm}
\usepackage{amsfonts}
\usepackage{comment}
\usepackage{mathrsfs}
\usepackage{hyperref}
\usepackage{titlesec}

\hypersetup{
    colorlinks=true,
    linkcolor=blue,
    citecolor=red,
    filecolor=magenta,      
    urlcolor=black,
    pdftitle={Periodic systems of subwavelength halide perovskite resonators},
    pdfpagemode=FullScreen,
    }
\usepackage{graphicx, tikz}
\usepackage{float}
\usepackage{xcolor}
\usepackage{bbm}
\usepackage[format=plain,
            font=it]{caption}

\newtheorem{theorem}{Theorem}[section]
\newtheorem{corollary}{Corollary}[theorem]
\newtheorem{lemma}[theorem]{Lemma}
\newtheorem{proposition}[theorem]{Proposition}
\newtheorem*{remark}{Remark}
\newtheorem*{notation}{Notation}
\newtheorem{definition}[theorem]{Definition}

\newcommand{\upd}{\mathrm{d}}  

\def\a{\alpha }     \def\b{\beta  }         \def\g{\gamma }
     \def\d{\delta}          \def\D{\Delta }
    \def\ve{\varepsilon}    
               
\def\k{\kappa }          \def\l{\lambda }    
\def\L{\Lambda }    \def\m{\mu}             \def\n{\nu}         
    \def\r{\rho}            \def\s{\sigma }     
     \def\t{\tau }

       \def\w{\omega }         \def\W{\Omega }

\newcommand{\axiscol}{gray}

\title{The effect of singularities and damping on the spectra of photonic crystals}
\author{Konstantinos Alexopoulos\thanks{Department of Mathematics, ETH Zurich, R\"amistrasse 101, CH-8092 Zurich, Switzerland.} \and Bryn Davies\thanks{Department of Mathematicsbb, Imperial College London, 180 Queen's Gate, London SW7 2AZ, United Kingdom.}}
\date{}
\begin{document}

\maketitle

\begin{abstract}
    Understanding the dispersive properties of photonic crystals is a fundamental and well-studied problem. However, the introduction of singular permittivities and damping complicates the otherwise straightforward theory. In this paper, we study photonic crystals with a Drude-Lorentz model for the permittivity, motivated by halide perovskites. We demonstrate how the introduction of singularities and damping affects the spectral band structure and show how to interpret the notion of a ``band gap" in this setting. We present explicit solutions for a one-dimensional model and show how integral operators can be used to handle multi-dimensional systems.
\end{abstract}

\tableofcontents

\section{Introduction}



Photonic crystals present interesting and useful wave properties. Even very simple photonic crystals, such as those composed of periodically alternating layers of non-dispersive materials, can display exotic dispersive properties. As a result, they are able to support band gaps: ranges of frequencies that are unable to propagate through the material \cite{soukoulis2012photonic}. These band gaps are the fundamental building blocks of the many different wave guides and wave control devices that have been conceived. Notable examples include flat lenses \cite{pendry2000negative}, invisibility cloaks \cite{milton2006cloaking}, rainbow trapping filters \cite{tsakmakidis2007trapped} and topological waveguides \cite{khanikaev2013photonic}.

When working at certain electromagnetic frequencies (which often includes the visible spectrum), it is important to take into account the oscillatory behaviour of the free electrons in a metal. This behaviour leads to resonances at characteristic frequencies and gives metals a highly dispersive character (even before the introduction of macroscopic structure, as in a photonic crystal). Several different models exist to describe this behaviour. Most models are variants of the Lorentz oscillator model, whereby electrons are modelled as damped harmonic oscillators due to electrostatic attractions with nuclei \cite{maier2007plasmonics}. A popular special case of this is the Drude model, in which case the restoring force is neglected (to reflect the fact that most electrons in metals are not bound to any specific nucleus, so lack a natural frequency of oscillation). Many other variants of these models exist, for instance by adding or removing damping from the various models, \emph{cf.} \cite{schulkin2004analytical} or \cite{li2003photonic}, and by taking linear combinations of the different models, as in \cite{sehmi2020applying}. 

A key feature that unites dispersive permittivity models is the existence of singularities in the permittivity. The position of these poles in the complex plane, which correspond to resonances, are one of the crucial properties that determines how a metal interacts with an electromagnetic wave. In conventional Lorentz models the poles appear in the lower complex plane \cite{maier2007plasmonics}. The imaginary part of the singular frequency is determined by the magnitude of the damping, and the singularities accordingly fall on the real line if the damping is set to zero. In the Drude model, the removal of the restorative force causes the singularities to fall at the origin and on the negative imaginary axis. 


A particularly important example of dispersive materials, that are central to the motivation for this study, are halide perovskites. They have excellent optical and electronic properties and are cheap and easy to manufacture at scale  \cite{MFTHBPZK}. As a result, they are being used in many applications, including optical sensors \cite{GPLLZZSQKJ}, solar cells \cite{S} and light-emitting diodes \cite{WKYZZPLBHL}. The dielectric permittivity of halide perovskites has been shown to depend heavily on excitonic transitions, leading to a permittivity that has symmetric poles in the lower complex plane \cite{MFTHBPZK}.



There are a range of methods that can be used to capture the spectra of photonic crystals. For one-dimensional systems, explicit solutions typically exist and transfer matrices are particularly convenient. These were used for Drude materials in \cite{sigalas1995metallic} and for undamped Lorentz materials in \cite{li2003photonic}, for example. In multiple dimensions, studies often resort to numerical simulation (for instance with finite elements). A valuable approximation strategy is a multi-scale asymptotic method known as high-frequency homogenisation \cite{craster2010high}, which can be extended to approximate the dispersion curves in dispersive media \cite{touboul2023dispersive}.

In this work, we will study photonic crystals composed of metals with permittivity inspired by that of halide perovskites, in the sense that it has symmetric poles in the lower complex plane. After setting out the Floquet-Bloch formulation of the periodic problem in section~\ref{sec:prob}, we will study the one-dimensional periodic Helmholtz problem in section~\ref{sec:1d}. We retrieve the dispersion relation which characterizes the halide perovskite system and show how its properties depend on the characteristics of the dispersive permittivity (namely, being real or being complex and having poles either on or below the real axis). Finally, in section~\ref{sec:multid}, we use integral methods and asymptotic analysis in order to obtain the dispersion relation for the two- and three-dimensional cases, showing how to extend our analysis to multi-dimensional photonic crystals.


\section{Problem setting} \label{sec:prob}

\subsection{Initial Helmholtz formulation}

Let us consider $N\in\mathbb{N}$ particles $D_1,D_2,\dots,D_N$ which together occupy a bounded domain $\Omega \subset \mathbb{R}^d$, for $d\in \{1,2,3\}$. The collection of particles $\Omega$ will be the repeating unit of the periodic photonic crystal. We suppose that permittivity of the particles is given by a Drude--Lorentz-type model, given by
\begin{align}\label{prmtvt}
    \ve(\w) = \ve_0 + \frac{\a}{1 - \b\w^2 - i \g \w},
\end{align}
where $\ve_0$ denotes the background dielectric constant and $\a,\b,\g$ are positive constants. $\a$ describes the strength of the interactions, $\b$ determines the natural resonant frequency and $\g$ is the damping factor. This is motivated by the measured permittivity of halide perovskites, as reported in \cite{MFTHBPZK}. We choose to use this expression as a canonical model for dispersive materials whose permittivities have singularities in the complex frequency space. Notice that \eqref{prmtvt} is singular at two complex values of $\omega$. These are given by
\begin{equation} \label{singularities}
    \omega^*_{\pm}= \frac{1}{2\b}\Big( -i\g \pm \sqrt{4\b-\g^2} \Big).
\end{equation}
By varying the parameters $\alpha$, $\beta$ and $\gamma$ we can force these singularities to lie in the lower half of the complex plane ($\gamma>0$), on the real line ($\gamma=0$) or to vanish completely ($\beta=\gamma=0$). We will make use of this property when trying to interpret the dispersion diagrams we obtain in the following analysis. We suppose that the particles are surrounded by a non-dispersive medium with permittivity $\ve_0$. We assume that the particles are non-magnetic, meaning the magnetic permeability $\m_0$ is constant on all of $\mathbb{R}^d$. 

We consider the Helmholtz equation as a model for the propagation of time-harmonic waves with frequency $\w$. This is a reasonable model for the scattering of transverse magnetic polarised light (see \emph{e.g.} \cite[Remark~2.1]{moiola2019acoustic} for a discussion). The wavenumber in the background $\mathbb{R}^d \setminus \overline{\W}$ is given by $k_0:=\w\ve_0\m_0$ and we will use $k$ to denote the wavenumber within $\W$. Let us note here that, from now on, we will suppress the dependence of $k_0$ and $k$ on $\w$ for brevity.  We, then, consider the system of equations
\begin{align}\label{Helmholtz problem}
    \begin{cases}
    \D u + \w^2 \ve(\w)\m_0 u = 0 \ \ \ &\text{ in } \W, \\
    \D u + k_0^2 u = 0 &\text{ in } \mathbb{R}^d\setminus\overline{\W}, \\
    u|_+ - u|_-=0 &\text{ on } \partial\W, \\
    \frac{\partial u}{\partial \n}|_+ - \frac{\partial u}{\partial \n}|_- = 0 &\text{ on } \partial\W, \\
    u(x) - u_{in}(x) &\text{ satisfies the outgoing radiation condition as } |x|\to \infty,\\
    \end{cases}
\end{align}
where $u_{in}$ is the incident wave, assumed to satisfy $(\D + k_0^2)u_{in}=0$, and the appropriate outgoing radiation condition depends on the dimension of the problem and of the periodic lattice.

\subsection{Periodic formulation}

We will assume that the collection of $N$ particles is repeated in a periodic lattice $\Lambda$. We suppose that the lattice has dimension $d_l$, in the sense that there are lattice vectors $l_1,\dots,l_{d_l}\in\mathbb{R}^d$ which generate $\Lambda$ according to
\begin{align}\label{L}
    \L := \{ m_1 l_1 + \dots + m_{d_l} l_{d_l} \ | \ m_i \in \mathbb{Z} \}.
\end{align}
The fundamental domain of the lattice $\L$ is the set $Y\in\mathbb{R}^{d}$ given by
\begin{align}
    Y := \{ c_1 l_1 + \dots + c_{d_l} l_{d_l} \ | \ 0 \leq c_1,\dots,c_{d_l} \leq 1 \}.
\end{align}
The dual lattice of $\Lambda$, denoted by $\Lambda^*$, is generated by the vectors $\a_1,\dots,\a_{d_l}$ satisfying $\a_i \cdot l_j = 2 \pi \d_{ij}$ for $i,j=1,\dots,d_l$. Finally, the \emph{Brillouin zone} $Y^*$ is defined by
\begin{equation}
    Y^* := (\mathbb{R}^{d_l} \times \{\mathbf{0}\}) / \L^*,
\end{equation}
where $\mathbf{0}$ is the zero vector in $\mathbb{R}^{d-d_l}$. The Brillouin zone $Y^*$ is the space that the reduced unit cell of reciprocal space.

\begin{figure}
\begin{center}
\includegraphics[scale=1.5]{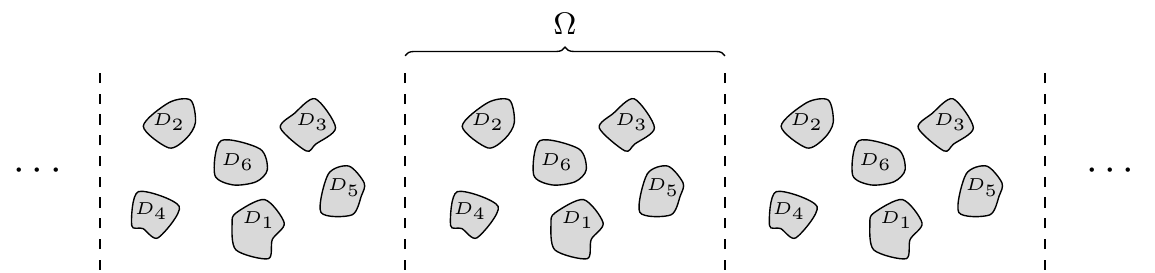}  
\end{center}
\caption{A periodic array of halide perovskite particles. Here, we have six particles $D_1,\dots,D_6$ repeated periodically in one dimension. Each of them has the halide perovskite permittivity $\ve(\w)$ defined by \eqref{prmtvt}.} \label{fig:periodic}
\end{figure}

The periodic structure, denoted by $\mathcal{D}$, is given by
\begin{align*}
    \mathcal{D} = \bigcup_{i=1}^N \left( \bigcup_{m\in\L} D_i + m \right).
\end{align*}
Hence, the problem we wish to study is the following:
\begin{align}\label{Periodic Helmholtz problem}
    \begin{cases}
    \D u + \w^2 \ve(\w)\m_0 u = 0 \ \ \ &\text{ in } \mathcal{D}, \\
    \D u + k_0^2 u = 0 &\text{ in } \mathbb{R}^d\setminus\overline{D}, \\
    u|_+ - u|_-=0 &\text{ on } \partial\mathcal{D}, \\
    \frac{\partial u}{\partial \n}|_+ - \frac{\partial u}{\partial \n}|_- = 0 &\text{ on } \partial\mathcal{D}, \\
    u(x_l,x_0) &\text{ satisfies the outgoing radiation condition as } |x_0|\to\infty.\\
    \end{cases}
\end{align}

\subsection{Floquet-Bloch theory}

In order to study the problem \eqref{Periodic Helmholtz problem}, we will make use of Floquet-Bloch theory \cite{kuchment2016overview}. Let us first give certain definitions which will help with the analysis of the problem. 

\begin{definition}
    A function $f(x)\in L^2(\mathbb{R}^d)$ is said to be $\k$-quasiperiodic, with quasiperiodicity $\k\in Y^*$, if $e^{-i\k\cdot x}f(x)$ is $\L$-periodic.
\end{definition}

\begin{definition}[Floquet transform]
    Let $f\in L^2(\mathbb{R}^d)$. The Floquet transform of $f$ is defined as
    \begin{align*}
        \mathcal{F}[f](x,\k) := \sum_{m\in\L} f(x-m) e^{-i\k\cdot x}, \ \ \ x\in\mathbb{R}^d, \ \k\in Y^*.
    \end{align*}
\end{definition}
We have that $\mathcal{F}[f]$ is $\k$-quasiperiodic in $x$ and periodic in $\k$. The Floquet transform is an invertible map $\mathcal{F}: L^2(\mathbb{R}^d)\rightarrow L^2(Y \times Y^*)$, with inverse given by
\begin{align*}
    \mathcal{F}^{-1}[g](x) = \frac{1}{|Y^*_l|} \int_{Y^*} g(x,\k) \upd \k, \ \ \ x\in\mathbb{R}^d,
\end{align*}
where $g(x,\k)$ is extended quasiperiodically for $x$ outside of the unit cell $Y$.

Let us define $u^{\k}(x):=\mathcal{F}[u](x,\k)$. Then, applying the Floquet transform to \eqref{Periodic Helmholtz problem}, we obtain the following system:
\begin{align}\label{Quasiperiodic Helmholtz problem}
    \begin{cases}
    \D u^{\k} + \w^2 \ve(\w)\m_0 u^{\k} = 0 \ \ \ &\text{ in } \mathcal{D}, \\
    \D u^{\k} + k_0^2 u^{\k} = 0 &\text{ in } \mathbb{R}^d\setminus\overline{D}, \\
    u^{\k}|_+ - u^{\k}|_-=0 &\text{ on } \partial\mathcal{D}, \\
    \frac{\partial u^{\k}}{\partial \n}|_+ - \frac{\partial u^{\k}}{\partial \n}|_- = 0 &\text{ on } \partial\mathcal{D}, \\
    u^{\k}(x_d,x_0) &\text{ is $\k$-quasiperiodic in $x_d$}, \\
    u^{\k}(x_d,x_0) &\text{ satisfies the $\k$-quasiperiodic radiation condition as } |x_0|\to\infty.\\
    \end{cases}
\end{align}

The solutions to \eqref{Quasiperiodic Helmholtz problem} typically take the form of a countable collection of spectral bands, each of which depends continuously on the Bloch parameter $\k$. The goal of our analysis is identifying and explaining the gaps between the spectral band. At frequencies within these band gaps, waves do not propagate in the material and their amplitude decays exponentially. As a result, they are the starting point for building waveguides and other wave control devices.

For real-valued permittivities, it is straightforward to define band gaps as the intervals between the real-valued bands:


\begin{definition}[Band gap for real permittivities] \label{BGreal}
    A frequency $\omega\in\mathbb{R}$ is said to be in a band gap of the periodic structure $\mathcal{D}$ if it is such that \eqref{Quasiperiodic Helmholtz problem} does not admit a non-trivial solution for any $\k\in\mathbb{R}$.
\end{definition}

We are interested in materials for which the permittivity takes complex values, corresponding to the introduction of damping to the model. In which case, we elect to keep the frequency $\omega\in\mathbb{R}$ as a real number but allow the Bloch parameter $\k$ to take complex values. In which case, the imaginary part of $\k$ describes the rate at which the waves amplitude decays. It should be noted that it is also quite common to do the opposite by forcing $\k$ to be real and allowing $\omega$ to be complex valued, as in \cite{ammari2022exceptional, touboul2023dispersive} for example.

In the real-valued case, it is clear that $\k$ belongs to the Brillouin zone $Y^*$ (which has the topology of a torus, due to the periodicity in $\k$). When $\k$ is complex valued, its real part still lives in $Y^*$ but its imaginary part can take arbitrary values. Thus, $\k$ lives in a space that is isomorphic to $Y^*\times\mathbb{R}$. 

In our setting, which is a damped model that is characterized by a complex permittivity, we have that $\k\in\mathbb{C}$ and it is less clear how to define a band gap. Intuitively, a band gap is a range of frequencies at which the damping is particularly large. Hence, we provide a modified definition for the notion of a band gap for complex permittivities in terms of local maxima of the amplitude decay:


\begin{definition}[Band gap for complex permittivities]
    We define a band gap for complex permittivities to be the set of frequencies $\w\in\mathbb{R}$ for which \eqref{Quasiperiodic Helmholtz problem} admits a non-trivial solution with quasiperiodicity $\k\in\mathbb{C}$ and $|\Im(\k)|$ is at a local maximum. 
\end{definition}

We will first study the problem in the one-dimensional setting. In one dimension, the problem is easier to manipulate and we are able to retrieve explicit expressions. Hence, we can get a variety of results concerning the characteristics of the quasiperiodic system. In particular, our main goal is to obtain the dispersion relation, an expression which relates the quasiperiodicities $\k\in \mathbb{C}$ with the frequencies $\w\in\mathbb{R}$, and study its properties. Then, in section \ref{sec:multid} we will provide an analysis for higher dimensional systems and we will give the equivalent relation.

\section{One dimension} \label{sec:1d}

\begin{figure}
\begin{center}
\includegraphics[scale=1.2]{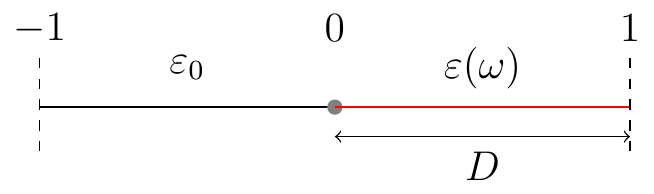}  
\end{center}
\caption{The one-dimensional setting. The periodically repeated cell is of length 2. Here the interval $[-1,0)$ is the background and the interval $[0,1)$ is the particle.}\label{fig:1D}
\end{figure}

Let us first treat the Helmholtz problem \eqref{Quasiperiodic Helmholtz problem} in the one-dimensional case. We will work on the interval $[-1,1]$, with $[-1,0)$ denoting the background and $[0,1)$ the particle. A schematic depiction of this is given in Figure \ref{fig:1D}. Hence, the problem reads as follows: 
\begin{align}\label{1D pb}
    \frac{\upd}{\upd x} \left( \frac{1}{\ve(x,\omega)} \frac{\upd u}{\upd x} \right) + \w^2 \m_0 u(x) = 0,
\end{align}
on the domain $[-1,1]$, where
\begin{align}\label{epsilon}
    \ve(x,\omega) := 
    \begin{cases}
        \ve_0, \ \ &x\in[-1,0) \ \ \text{(background)}, \\
        \ve(\w), \ \ &x\in[0,1] \ \ \quad \text{(particle)},
    \end{cases}
\end{align}
with the boundary conditions
\begin{align}\label{bd1}
    u(1) = e^{2i\k}u(-1)\quad\text{and}\quad \frac{\upd u}{\upd x}(1) = e^{2i\k}\frac{\upd u}{\upd x}(-1).
\end{align}

\subsection{Dispersion relation}

We will now retrieve an expression for the solution to \eqref{1D pb}. Let us define the quantities
\begin{align}\label{def disp}
    \s_0 := \w\sqrt{\ve_0 \m_0}, \quad\text{and}\quad \s_c := \w\sqrt{\ve(\w) \m_0}. 
\end{align}
For many of the results that follow, the crucial quantity will be the contrast between the material inside the particles and the background medium. With this in mind, we introduce the frequency-dependent contrast $\rho$ as
\begin{equation} \label{def:contrast}
    \r(\omega) := \frac{\s_c}{\s_0} = \sqrt{\frac{\ve(\w)}{\ve_0}}.
\end{equation}
Then, the following expression holds for the solution to \eqref{1D pb}.

\begin{lemma}\label{almost explicit solution}
    Let $u$ denote a solution to \eqref{1D pb}. Then, $u$ is given by
    \begin{align}\label{u cst simple}
    u(x) = 
    \begin{cases}
        \mathcal{A} \r \sin\Big( \s_0 x \Big) +\mathcal{B} \cos\Big( \s_0 x \Big) , &x\in[-1,0) \\
        \mathcal{A} \sin\Big( \s_c x \Big) +\mathcal{B} \cos\Big( \s_c x \Big) , &x\in[0,1],
    \end{cases}
\end{align}
where $ \mathcal{A},\mathcal{B}\in\mathbb{C}$ are two constants.
\end{lemma}

\begin{proof}
    We know that a solution to \eqref{1D pb} must be given by
    \begin{align}\label{u cst}
        u(x) = 
        \begin{cases}
        \mathcal{A}_1 \sin\Big( \w\sqrt{\ve_0\m_0} x \Big) +\mathcal{B}_1 \cos\Big( \w\sqrt{\ve_0\m_0} x \Big) , \ \ \ \ &x\in[-1,0), \\
        \mathcal{A}_2 \sin\Big( \w\sqrt{\ve(\w)\m_0} x \Big) +\mathcal{B}_2 \cos\Big( \w\sqrt{\ve(\w)\m_0} x \Big) , \ \ \ \ &x\in[0,1],
    \end{cases}
    \end{align}
    where $ \mathcal{A}_1, \mathcal{A}_2, \mathcal{B}_1, \mathcal{B}_2 \in \mathbb{C}$ are constants to be defined. This, also, gives
    \begin{align*}
    \frac{\upd u}{\upd x}(x) = 
    \begin{cases}
        \mathcal{A}_1 \w\sqrt{\ve_0\m_0} \cos\Big( \w\sqrt{\ve_0\m_0} x \Big) - \mathcal{B}_1 \w\sqrt{\ve_0\m_0} \sin\Big( \w\sqrt{\ve_0\m_0} x \Big) , \ \ \ \ &x\in[-1,0), \\
        \mathcal{A}_2 \w\sqrt{\ve(\w)\m_0} \cos\Big( \w\sqrt{\ve(\w)\m_0} x \Big) - \mathcal{B}_2 \w\sqrt{\ve(\w)\m_0} \sin\Big( \w\sqrt{\ve(\w)\m_0} x \Big) , \ \ \ \ &x\in[0,1].
    \end{cases}
    \end{align*}
    Now, from the boundary transmission conditions in \eqref{Periodic Helmholtz problem}, we require
    \begin{align*} 
        \lim_{x\to0^{-}} u(x) = \lim_{x\to0^{+}} u(x) \quad \quad \text{ and } \quad \quad \lim_{x\to0^{-}} \frac{\upd u}{\upd x}(x) = \lim_{x\to0^{+}} \frac{\upd u}{\upd x}(x).
    \end{align*}
    These conditions mean we must have that $\mathcal{B}_1 = \mathcal{B}_2$ and $\mathcal{A}_1 = \sqrt{{\ve(\w)}/{\ve_0}} \mathcal{A}_2 = \r(\w) \mathcal{A}$, 
    which gives the desired result.
\end{proof}

Using the boundary conditions \eqref{bd1}, we can obtain the dispersion relation for the one-dimensional problem. This is a well-known result, that first appeared in a quantum-mechanical setting \cite{kronig1931quantum} and has since been shown to describe a range of periodic classical wave systems also \cite{adams2008bloch, movchan2002asymptotic}. We include a brief proof, for completeness.

\begin{theorem}[Dispersion relation]
    Let $u$ denote the solution to \eqref{1D pb} along with the boundary conditions \eqref{bd1}. Then, for $u$ to be non-trivial, the quasiperiodicities $\k\in \mathbb{C}$ satisfies the dispersion relation
    \begin{align}\label{1D dispersion relation}
        \cos(2\k) = \cos(\s_0) \cos(\r \s_0) - \frac{1+\r^2}{2\r} \sin(\s_0) \sin(\r\s_0).
    \end{align}
\end{theorem}

\begin{proof}
    From Lemma \ref{almost explicit solution}, we have that $u$ is given by \eqref{u cst simple}. Then, using \eqref{bd1}, we have
\begin{align}\label{system cst}
    \begin{cases}
        \Big[ \sin( \s_c ) + e^{2i\k}\r\sin(\s_0) \Big] \mathcal{A} + \Big[ \cos( \s_c ) - e^{2i\k}\cos(\s_0) \Big] \mathcal{B} = 0, \\
        \Big[ \s_c\cos(\s_c) -e^{2i\k}\r\s_0\cos(\s_0) \Big]\mathcal{A} - \Big[ \s_c\sin(\s_c) + e^{2i\k}\s_0\sin(\s_0) \Big]\mathcal{B} = 0.
    \end{cases}
\end{align}
We observe that for \eqref{1D pb} to have a non-zero solution, it should hold
\begin{align*}
    \Big[ \sin( \s_c ) + e^{2i\k}\r\sin(\s_0) \Big] \cdot \Big[ \s_c &\sin(\s_c) + e^{2i\k}\s_0\sin(\s_0) \Big] + \\
    &+ \Big[ \cos( \s_c ) - e^{2i\k}\cos(\s_0) \Big] \cdot \Big[ \s_c\cos(\s_c) -e^{2i\k}\r\s_0\cos(\s_0) \Big] = 0,
\end{align*}
which gives
\begin{align}\label{eq:2}
    \sqrt{\ve(\w)} e^{4i\k} + \left[ \frac{\ve_0+\ve(\w)}{\sqrt{\ve_0}}\sin(\s_0)\sin(\s_c) - 2 \sqrt{\ve(\w)} \cos(\s_0) \cos(\s_c) \right] e^{2i\k} + \sqrt{\ve(\w)} = 0.
\end{align}
Making some algebraic rearrangements, we observe that 
\begin{align}\label{eq:3}
\begin{aligned}   
    2 \sqrt{\frac{\ve(\w)}{\ve_0}}\Big[ \cos(\s_0) \cos(\s_c) - \cos(2\k) \Big] - \frac{\ve_0+\ve(\w)}{\ve_0} \sin(\s_0) \sin(\s_c) = 0.
\end{aligned}
\end{align}
Finally, making the substitutions $\s_c = \r \s_0$ and $\sqrt{\ve(\w)} = \r \sqrt{\ve_0}$, we obtain the desired result. 
\end{proof}

The dispersion relation \eqref{1D dispersion relation} can be used to plot the dispersion curves. For a given frequency $\omega$, $\rho(\omega)$ can be calculated to yield the right hand side of \eqref{1D dispersion relation}, which can subsequently be solved to find $\kappa$. This is shown in Figure~\ref{fig:HPdispersion}. Since $\epsilon(\omega)$ is complex valued, $\kappa$ will generally take complex values. We plot only the absolute values of both the real and imaginary parts; as we will see below, this is sufficient to characterise the full dispersion relation. Notice also that $\Re(\k)\in Y^*=[-\pi/2,\pi/2)$.

\begin{figure} 
    \centering
    \includegraphics[width=0.6\linewidth]{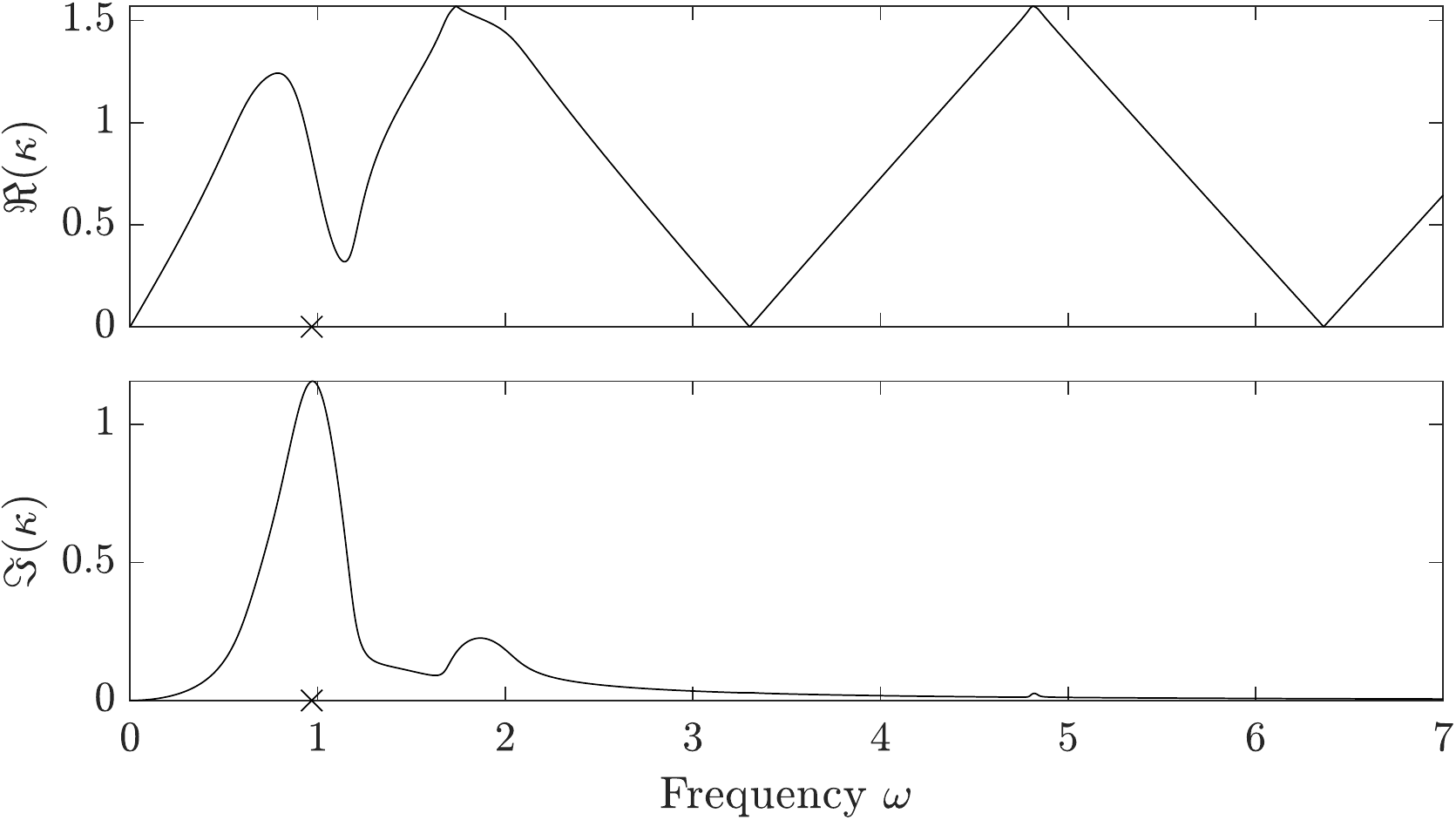}
    \begin{tikzpicture}
    \node[anchor=west] at (-1.75,1.5) {Singularities:};
    \draw[<->,\axiscol] (-1.5,0) -- (1.5,0);
    \draw[<->,\axiscol] (0,-1) -- (0,1);
    \node[anchor=north,scale=0.6,\axiscol] at (1.5,0) {Re$(\omega)$};
    \node[anchor=east,scale=0.6,\axiscol] at (0,1) {Im$(\omega)$};
    \node[anchor=west] at (0.5,-0.5) {$\times \, \omega^*_+$};
    \node[anchor=east] at (-0.5,-0.5) {$\omega^*_- \, \times$};

    \draw[\axiscol] (0.78,0.1) -- (0.78,-0.1) node[at end,anchor=north,scale=0.5]{0.97};
    \draw[\axiscol] (-0.78,0.1) -- (-0.78,-0.1) node[at end,anchor=north,scale=0.5]{-0.97};
    \draw[\axiscol] (-0.1,-0.55) -- (0.1,-0.55) node[at end,anchor=west,scale=0.5]{-0.25};
    
    \node[white] at (-2,-3) {.}; 
    \end{tikzpicture}
    \caption{The dispersion relation of the halide perovskite photonic crystal. We model a material with permittivity given by \eqref{prmtvt} with $\alpha=1$, $\beta=1$ and $\gamma=0.5$. The frequency $\omega$ is chosen to be real and the Bloch parameter $\kappa$ allowed to take complex values. The permittivity is singular at two points, which are in the lower complex plane and are symmetric about the imaginary axis, as indicated in the sketch on the right and (the real parts) by the crosses on the frequency axes of the plots. }
    \label{fig:HPdispersion}
\end{figure}
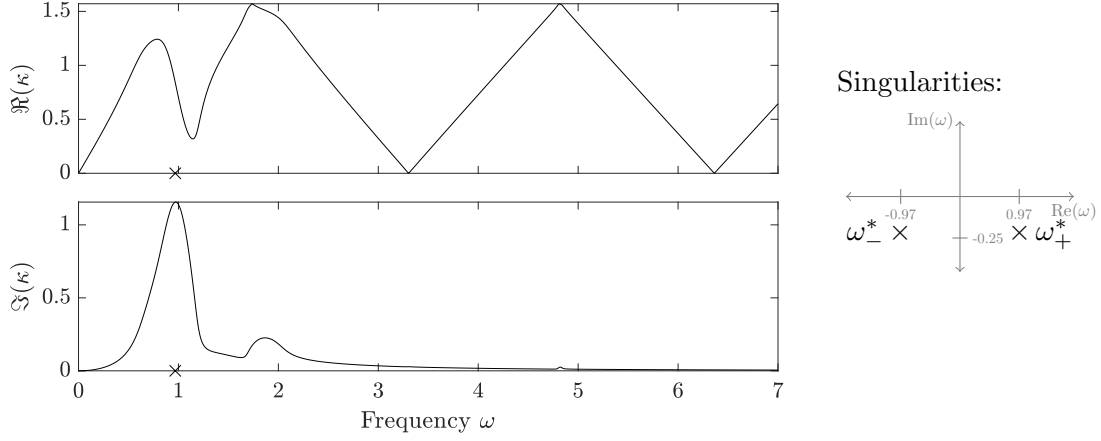

\subsection{Properties of the dispersion relation}\label{Properties of the dispersion relation}

The dispersion relation \eqref{1D dispersion relation} describes the behaviour of the periodic system and reveals the relationship between the quasiperiodicities $\k\in \mathbb{C}$, the frequencies $\w\in\mathbb{R}$ and the permittivity $\ve(\w)$ of the material. We can use it to derive some simple results about the dispersion curves. The first thing to understand is the symmetries of the dispersion curves.

\begin{lemma}[Opposite quasiperiodicities]\label{Opposite quasiperiodicities}
    Let $\k\in \mathbb{C}$ be a complex quasiperiodicity satisfying the dispersion relation \eqref{1D dispersion relation} for a given frequency $\omega\in\mathbb{R}$. Then, the opposite quasiperiodicity, i.e. $-\k$, satisfies the same dispersion relation.
\end{lemma}
\begin{proof}
    We just have to use that $\cos(\cdot)$ is an even function. Then, if $\k\in\mathbb{C}$ is such that \eqref{1D dispersion relation} holds, from the fact that $\cos(-2\k) = \cos(2\k)$, we get that $-\k\in\mathbb{C}$ also satisfies \eqref{1D dispersion relation}. This concludes the proof.
\end{proof}

It is with Lemma~\ref{Opposite quasiperiodicities} in mind that we are able to plot only the absolute values of the imaginary parts in Figure~\ref{fig:HPdispersion} and the subsequent figures.

\subsubsection{Real and imaginary parts}

In the analysis that will follow, it will be useful to be able to describe the behaviour of the real and imaginary part of the quasiaperiodicity with respect to the permittivity. In particular, we will decompose both the quasiperiodicity $\k$ and $\r$ into real and imaginary parts, and we will derive this dependence from the dispersion relation. Since $\k\in\mathbb{C}$ and $\r\in\mathbb{C}$, let us define:
\begin{align}\label{a complex}
    \k = \k_1 + i\k_2 \ \ \ \text{ and } \ \ \ \r = \r_1 + i\r_2,
\end{align}
with $\k_1,\k_2,\r_1,\r_2\in\mathbb{R}$. We will also define $\mathscr{L}_1$ and $\mathscr{L}_2$, which depend on $\w\in\mathbb{R}$, as follows:
\begin{align}\label{L1}
\begin{aligned}
    \mathscr{L}_1(\omega):=&\cos(\s_0) \cos(\s_0\r_1)\cosh(\s_0\r_2) - \frac{\sin(\s_0)}{2(\r_1^2+\r_2^2)}\Big[ \r_1(1+\r_1^2+\r_2^2)\sin(\s_0\r_1)\cosh(\s_0\r_2)-\\
    &-\r_2(\r_2^2-1+\r_1^2) \cos(\s_0\r_1) \sinh(\s_0\r_2)\Big],
\end{aligned}
\end{align}
and
\begin{align}\label{L2}
\begin{aligned}
    \mathscr{L}_2(\omega):=&\cos(\s_0) \sin(\s_0\r_1) \sinh(\s_0\r_2) + \frac{\sin(\s_0)}{2(\r_1^2+\r_2^2)}\Big[\r_2(\r_2^2-1+\r_1^2)\sin(\s_0\r_1)\cosh(\s_0\r_2) +\\
    &+ \r_1(1+\r_1^2+\r_2^2)\cos(\s_0\r_1)\sinh(\s_0\r_2) \Big],
\end{aligned}
\end{align}
where we note that $\rho_1$, $\rho_2$ and $\sigma_0$ all depend on the frequency $\omega$, as specified in \eqref{def disp} and \eqref{def:contrast}. Then, we have the following result.

\begin{proposition}\label{real and im}
    Let $\k\in\mathbb{C}$, given by \eqref{a complex}, satisfying the dispersion relation \eqref{1D dispersion relation} for a given frequency $\omega\in\mathbb{R}$. Then, its real and imaginary parts are given by
    \begin{align}\label{a_1}
        \Re(\k) = \pm \frac{1}{2} \arccos\left(\frac{\mathscr{L}_1}{\cosh(2\Im(\k))}\right),
    \end{align}
    and
    \begin{align}\label{a_2}
        \Im(k) = \frac{1}{2} \text{\emph{arcsinh}}\left( \pm \sqrt{ \frac{1}{2} \Big[\mathscr{L}_1^2+\mathscr{L}_2^2-1+\sqrt{(1-\mathscr{L}_1^2-\mathscr{L}_2^2)^2+4\mathscr{L}_2^2}\Big] } \right),
    \end{align}
    where $\mathscr{L}_1$ and $\mathscr{L}_2$ are given by \eqref{L1} and \eqref{L2}, respectively. We, also, note that the choice of $+$ or $-$ should be the same in \eqref{a_1} and \eqref{a_2}.
\end{proposition}

\begin{proof}
    From \eqref{a complex}, the dispersion relation \eqref{1D dispersion relation} becomes
    \begin{align*}
        \cos(2\k_1+i2\k_2) = \cos(\s_0)\cos(\s_0\r_1+i\s_0\r_2) - \frac{1+(\r_1+i\r_2)^2}{2(\r_1+i\r_2)}\sin(\s_0)\sin(\s_0\r_1+i\s_0\r_2),
    \end{align*}
    which is,
    \begin{align*}
        \cos(2\k_1)&\cosh(2\k_2) - i \sin(2\k_1)\sinh(2\k_2) = \\
        &=\cos(\s_0)\Big[ \cos(\s_0\r_1)\cosh(\s_0\r_2) - i \sin(\s_0\r_1) \sinh(\s_0\r_2) \Big] -\\
        &\quad- \frac{1+\r_1^2+2i\r_1\r_2-\r_2^2}{2(\r_1^2+\r_2^2)}(\r_1-i\r_2)\sin(\s_0)\Big[ \sin(\s_0\r_1)\cosh(\s_0\r_2) + i\cos(\s_0\r_1)\sinh(\s_0\r_2) \Big].
    \end{align*}
    Taking real and imaginary parts, we obtain, for the real part,
    \begin{align}\label{real part condition}
            \cos(2\k_1)\cosh(2\k_2) = &\cos(\s_0) \cos(\s_0\r_1)\cosh(\s_0\r_2) - \frac{\sin(\s_0)}{2(\r_1^2+\r_2^2)}\Big[ \r_1(1+\r_1^2+\r_2^2)\sin(\s_0\r_1)\cosh(\s_0\r_2)+\nonumber \\
            &-\r_2(\r_2^2-1+\r_1^2) \cos(\s_0\r_1) \sinh(\s_0\r_2)\Big],
    \end{align}
    and, for the imaginary part,
    \begin{align}\label{imaginary part condition}
        \sin(2\k_1)\sinh(2\k_2) = &\cos(\s_0) \sin(\s_0\r_1) \sinh(\s_0\r_2) + \frac{\sin(\s_0)}{2(\r_1^2+\r_2^2)}\Big[\r_2(\r_2^2-1+\r_1^2)\sin(\s_0\r_1)\cosh(\s_0\r_2) +\nonumber \\
        &+ \r_1(1+\r_1^2+\r_2^2)\cos(\s_0\r_1)\sinh(\s_0\r_2) \Big].
    \end{align}
    So, from \eqref{L1} and \eqref{L2}, we obtain the system
    \begin{align}\label{system with L}
        \begin{cases}
            \cos(2\k_1)\cosh(2\k_2) = \mathscr{L}_1,\\
            \sin(2\k_1)\sinh(2\k_2) = \mathscr{L}_2.
        \end{cases}
    \end{align}
    From the first equation, we immediately see that
    \begin{align*}
        \k_1 =\pm \frac{1}{2} \arccos\left(\frac{\mathscr{L}_1}{\cosh(2\k_2)}\right),
    \end{align*}
    then, substituting into the second equation gives
    \begin{align*}
        \sin\left[\arccos\left(\frac{\mathscr{L}_1}{\cosh(2\k_2)}\right)\right]\sinh(2\k_2) = \pm\mathscr{L}_2.
    \end{align*}
    We know that for $x\in[-1,1]$, we have the identity $\sin[\arccos(x)] = \sqrt{1-x^2}$. Hence, from the above, we get
    \begin{align}\label{dagger}
        \sqrt{1 - \frac{\mathscr{L}_1^2}{\cosh^2(2\k_2)} } \sinh(2\k_2) = \pm\mathscr{L}_2.
    \end{align}
    Similarly, using the fact that $\cosh^2(x) - \sinh^2(x) = 1$ for $x\in\mathbb{R}$, we find that 
    \begin{align}\label{dagger 2}
        \sqrt{1 - \frac{\mathscr{L}_1^2}{1+\sinh^2(2\k_2)} } \sinh(2\k_2) = \pm\mathscr{L}_2.
    \end{align}
    Hence, we have
    \begin{align*}
        \sinh^4(2\k_2) + (1-\mathscr{L}_1^2-\mathscr{L}_2^2) \sinh^2(2\k_2) - \mathscr{L}_2^2 = 0.
    \end{align*}
    Using the quadratic formula, this gives
    \begin{align*}
        \sinh^2(2\k_2) = \frac{1}{2}\Big[\mathscr{L}_1^2+\mathscr{L}_2^2-1+\sqrt{(1-\mathscr{L}_1^2-\mathscr{L}_2^2)^2+4\mathscr{L}_2^2}\Big],
    \end{align*}
    and so, we get
    \begin{align*}
        \k_2 = \frac{1}{2} \text{arcsinh}\left( \pm \sqrt{ \frac{1}{2} \Big[\mathscr{L}_1^2+\mathscr{L}_2^2-1+\sqrt{(1-\mathscr{L}_1^2-\mathscr{L}_2^2)^2+4\mathscr{L}_2^2}\Big] } \right).
    \end{align*}
    This gives the desired result.
\end{proof}

\begin{remark}
    Another way of viewing that the choice of $+$ or $-$ in \eqref{a_1} is the same as the one in \eqref{a_2} is from the fact that we have shown that if $\k\in\mathbb{C}$ satisfies \eqref{1D dispersion relation}, then $-\k$ does as well, but $\overline{\k}$ does not.
\end{remark}

\subsubsection{Imaginary part decay}

From \eqref{a_2}, we obtain a result on the decay of the imaginary part of the quasiperiodicity $\k$ as $\w\to\infty$. We will first state some preliminary results, before proving the main theorem.

\begin{lemma}\label{r limit}
    Let the frequency-dependent contrast $\r\in\mathbb{C}$ be given by \eqref{def:contrast}. Then, it holds
    \begin{align}\label{limit r}
        \lim_{\w\to\infty} |\r| = 1,
    \end{align}
    and 
    \begin{align}\label{limit r_1 and r_2}
        \lim_{\w\to\infty} |\Re(\r)| = 1 \ \ \ \text{ and } \ \ \ \lim_{\w\to\infty} |\Im(\r)| = 0.
    \end{align}
\end{lemma}

\begin{proof}
    From \eqref{def disp}, we have
    \begin{align*}
        \r &= \sqrt{1 + \frac{\a}{\ve_0(1-\b\w^2 - i\g\w)}} 
    \end{align*}
    which gives directly $\lim_{\w\to\infty} |\r| = 1$. This can be rewritten as
    \begin{align*}
        \r &= \sqrt{1 + \frac{\a}{\ve_0}\frac{1-\b\w^2}{(1-\b\w^2)^2 + \g^2 \w^2} + i \frac{\a}{\ve_0} \frac{\g\w}{(1-\b\w^2)^2 + \g^2 \w^2}}
    \end{align*}
    To ease the notation, let us write
    \begin{align}\label{a and b}
        a(\w) := 1 + \frac{\a}{\ve_0}\frac{1-\b\w^2}{(1-\b\w^2)^2 + \g^2 \w^2} \ \ \ \text{ and } \ \ \ b(\w) := \frac{\a}{\ve_0} \frac{\g\w}{(1-\b\w^2)^2 + \g^2 \w^2}.
    \end{align}
    Then, we have
    \begin{align}\label{r as a+ib}
        \r = \pm \left( \sqrt{\frac{\sqrt{a(\w)^2 + b(\w)^2} + a(\w)}{2}} + i \frac{b(\w)}{|b(\w)|} \sqrt{\frac{\sqrt{a(\w)^2 + b(\w)^2} - a(\w)}{2}} \right)
    \end{align}
    We observe, from \eqref{a and b}, that, as $\w\to\infty$,
    \begin{align}\label{a and b asymptotic behaviour}
        a(\w) = 1 + O\left(\frac{1}{\w^2}\right) \ \ \ \text{ and } \ \ \ b(\w) = O\left(\frac{1}{\w^3}\right),
    \end{align}
    which gives
    \begin{align*}
        \lim_{\w\to\infty} a(\w) = 1 \ \ \ \text{ and } \ \ \ \lim_{\w\to\infty} b(\w) = 0.
    \end{align*}
    Also, since $\a,\g,\ve_0>0$, it holds that
    \begin{align*}
        \lim_{\w\to\infty} \frac{b(\w)}{|b(\w)|} = 1.
    \end{align*}
    Hence, combining these results, we get
    \begin{align*}
        \lim_{\w\to\infty} |\r_1| = \lim_{\w\to\infty} \left| \sqrt{\frac{\sqrt{a(\w)^2 + b(\w)^2} + a(\w)}{2}} \right| = 1
    \end{align*}
    and
    \begin{align*}
        \lim_{\w\to\infty} |\r_2| = \lim_{\w\to\infty} \left| \frac{b(\w)}{|b(\w)|} \sqrt{\frac{\sqrt{a(\w)^2 + b(\w)^2} - a(\w)}{2}} \right| = 0.
    \end{align*}
    This concludes the proof.
\end{proof}

\begin{lemma}\label{limit L1 L2}
    As $\w\to\infty$, we have that
    \begin{align}\label{L_i estimate}
        |\mathscr{L}_1| \leq 1 \ \ \ \text{ and } \ \ \ \mathscr{L}_2 \to 0,
    \end{align}
    where $\mathscr{L}_1=\mathscr{L}_1(\omega)$ and $\mathscr{L}_2=\mathscr{L}_2(\omega)$ were defined in \eqref{L1} and \eqref{L2}.
\end{lemma}
\begin{proof}
    From Lemma \ref{r limit}, we have that, as $\w\to\infty$,
    \begin{align*}
        |\r_1|\to1, \ \ \ |\r_2|\to0
    \end{align*}
    and from \eqref{def disp}, we have that 
    \begin{align*}
        \s_0 \to \infty.
    \end{align*}
    So, it is essential to understand the behaviour of $\s_0\r_2$ as $\w\to\infty$. Using the same notations as in the proof of Lemma \ref{r limit}, we have, without loss of generality on the $\pm$ of \eqref{r as a+ib},
    \begin{align*}
        \r_2 = \frac{b(\w)}{|b(\w)|} \sqrt{\frac{\sqrt{a(\w)^2 + b(\w)^2} - a(\w)}{2}},
    \end{align*}
    and hence, from \eqref{def disp}, we have
    \begin{align*}
        \s_0 \r_2 = \sqrt{\m_0\ve_0} \frac{b(\w)}{|b(\w)|} \sqrt{\frac{\sqrt{\w^4\Big(a(\w)^2 + b(\w)^2\Big)} - \w^2 a(\w)}{2}}.
    \end{align*}
    From \eqref{a and b asymptotic behaviour}, we see that, as $\w\to\infty$,
    \begin{align*}
        \w^4 a(\w)^2 = \w^4 + O(1), \ \ \ \w^2 a(\w) = \w^2 + O(1) \ \ \ \text{ and } \ \ \ \w^4 b(\w)^2 = O\left(\frac{1}{\w^2}\right).
    \end{align*}
    Hence, as $\w\to\infty$,
    \begin{align*}
        \sqrt{\frac{\sqrt{\w^4\Big(a(\w)^2 + b(\w)^2\Big)} - \w^2 a(\w)}{2}} \to 0,
    \end{align*}
    which gives,
    \begin{align*}
        \lim_{\w\to\infty} \s_0 \r_2 = 0,
    \end{align*}
    and so
    \begin{align}\label{limit cosh sinh}
        \lim_{\w\to\infty} |\cosh(\s_0\r_2)| = 1 \ \ \ \text{ and } \ \ \ \lim_{\w\to\infty} |\sinh(\s_0\r_2)| = 0.
    \end{align}
    Thus, \eqref{L1} gives
    \begin{align*}
        \lim_{\w\to\infty}|\mathscr{L}_1| &= \lim_{\w\to\infty}\Big|\cos(\s_0)\cos(\s_0 \r_1) - \sin(\s_0)\sin(\s_0 \r_1)\Big|\\
        &=\lim_{\w\to\infty} \Big|\cos\Big(\s_0(1+\r_1)\Big)\Big| \leq 1,
    \end{align*}
    which is the desired bound for $\mathscr{L}_1$. Similarly, from the triangle inequality applied on \eqref{L2}, we have
    \begin{align}\label{L2 bound}
    \begin{aligned}
        |\mathscr{L}_2| \leq |\sinh(\s_0\r_2)| + \frac{1}{2(\r_1^2+\r_2^2)}\Big[ |\r_2&(\r_2^2-1+\r_1^2)||\cosh(\s_0\r_2)| +\\
        &+ |\r_1|(1+\r_1^2+\r_2^2)|\sinh(\s_0\r_2)| \Big].
    \end{aligned}
    \end{align}
    Using Lemma \ref{r limit} and \eqref{limit cosh sinh}, we obtain
    \begin{align*}
        \lim_{\w\to\infty} \mathscr{L}_2 = 0.
    \end{align*}
    This concludes the proof.
\end{proof}

Using these results, we will describe the behaviour of the imaginary part $\k_2$ of the quasiperiodicity $\k\in\mathbb{C}$ as the frequency tends to infinity, i.e. $\w\to\infty$.

\begin{proposition}\label{im decay}
    Let us consider a complex quasiperiodicity $\k\in \mathbb{C}$ satisfying the dispersion relation \eqref{1D dispersion relation} with $\a,\b,\g\in\mathbb{R}_{>0}$. Then, it holds that
    \begin{align}\label{limi a_2}
        \lim_{\w\to\infty} \Im(\k) = 0.
    \end{align}
\end{proposition}

\begin{proof}
    Indeed, since $\k\in\mathbb{C}$, let us define $\k_1:=\Re(\k)$ and $\k_2:=\Im(\k)$. Then, from \eqref{a_2}, we have that $\k_2$ is given by
    \begin{align*}
        \k_2 = \frac{1}{2} \text{arcsinh}\left( \pm \sqrt{ \frac{1}{2} \Big[\mathscr{L}_1^2+\mathscr{L}_2^2-1+\sqrt{(1-\mathscr{L}_1^2-\mathscr{L}_2^2)^2+4\mathscr{L}_2^2}\Big] } \right)
    \end{align*}
    From Lemma \ref{limit L1 L2}, we have that, as $\w\to\infty$, $\mathscr{L}_1$ remains bounded, whereas $\mathscr{L}_2\to0$. Thus, the following holds
    \begin{align*}
        \lim_{\w\to\infty}\Big[\mathscr{L}_1^2+\mathscr{L}_2^2-1+\sqrt{(1-\mathscr{L}_1^2-\mathscr{L}_2^2)^2+4\mathscr{L}_2^2}\Big] &=  \Big[\mathscr{L}_1^2-1+|1-\mathscr{L}_1^2|\Big] =0,
    \end{align*}
    since we have that $|\mathscr{L}_1|\leq1$ also from Lemma \ref{limit L1 L2}. Then, from the continuity of the $\mathrm{arcsinh}(\cdot)$ function, the desired result follows.
\end{proof}

The decay predicted by Proposition \ref{im decay} is shown in Figure \ref{fig:HPdispersion}. Due to the damping in the model, the imaginary part has discernible peaks at the first few gaps, but then decays steadily to zero at higher frequencies.

\subsection{The effect of singularities and damping}

As mentioned before, the dispersion relation of the halide perovskite particles leads to dispersion curves which are not trivial to understand in terms of the traditional viewpoint of band gaps. In order to understand the behaviour, we will examine each distint feature of the halide perovskite permittivity, to understand the effect it has on the spectrum of the photonic crystal.

In particular, we will begin with the simplest case when the permittivity is real and constant with respect to the frequency $\w$. Then, we will introduce a dispersive behaviour to the permittivity by adding singularities at non-zero frequencies. We will initially suppose that these poles lie on the real axis and will study the behaviour close to these regions. Finally, we will study the effect of introducing a complex permittivity, corresponding to damping. Taken together, these results will allow us to explain the spectra the halide perovskite photonic crystal.

\subsubsection{Constant permittivity} \label{sec:realconst}

The first case we will consider is the one of a real-valued, non-dispersive permittivity, constant with respect to the frequency $\w$. In our setting this translates into having \eqref{prmtvt} with $\b=\g=0$ and $\a>0$, i.e.
\begin{align}\label{e_cst}
    \ve(\w) = \ve_0 + \a \in \mathbb{R}_{>0}.
\end{align}
This setting has been studied quite extensively. We refer to \cite{brillouin1953wave}, as a classical reference for studying the dispersive nature of waves in periodic systems. In Figure~\ref{fig:plot:e_cst}, we provide an example for the dispersion curves when the permittivity is constant and real. It is worth noting the following result.

\begin{figure} 
    \centering
    \includegraphics[width=0.6\linewidth]{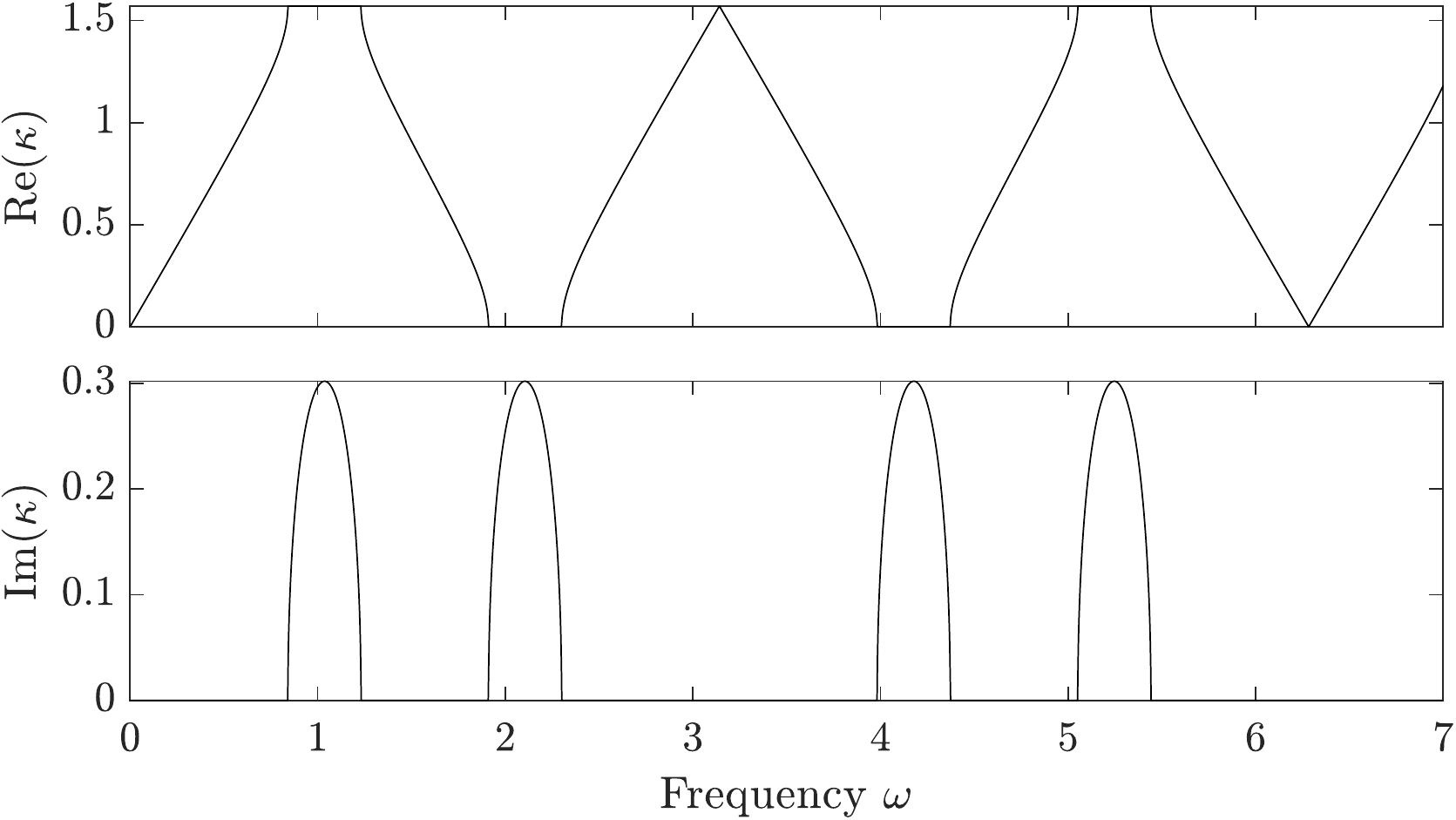}
    \caption{The dispersion relation of a photonic crystal with frequency-independent material parameters. We model a material with permittivity given by \eqref{prmtvt} with $\alpha=1$, $\beta=0$ and $\gamma=0$. The frequency $\omega$ is chosen to be real and the Bloch parameter $\kappa$ allowed to take complex values. The permittivity is never singular in this case. }
    \label{fig:plot:e_cst}
\end{figure}

\begin{lemma} \label{lem:realsymm}
    Let $\ve(\w)$ be the real-valued, non-dispersive permittivity given by \eqref{e_cst}. Then, if $\k\in\mathbb{C}$ is a quasiperiodicity satisfying \eqref{1D dispersion relation} for a given frequency $\w\in\mathbb{R}$, then so does $\overline{\k}\in\mathbb{C}$.
\end{lemma}

\begin{proof}
    Indeed, since $\ve(\w)\in\mathbb{R}_{>0}$, then $\r\in\mathbb{R}_{>0}$. Let us take $\k\in\mathbb{C}$ satisfying \eqref{1D dispersion relation}. Then, we can write $\k=\k_1+i\k_2$, with $\k_1,\k_2\in\mathbb{R}$. Now, since $\r>0$, we can write
    \begin{align*}
        \cos(\s_0) \cos(\r\s_0) - \frac{1+\r^2}{2\r} \sin(\s_0)\sin(\r\s_0) =: A>0.
    \end{align*}
    Thus, \eqref{1D dispersion relation} gives us
    \begin{align*}
        \cos(2\k_1+2i\k_2) = A,
    \end{align*}
    which becomes the following system
    \begin{align*}
        \begin{cases}
            \cos(2\k_1) \cdot \cosh(2\k_2) = A, \\
            \sin(2\k_1) \cdot \sinh(2\k_2) = 0.
        \end{cases}
    \end{align*}
    This implies that
    \begin{align*}
        \k_2 = 0 \ \ \ \text{ or } \ \ \ \k_1 = \frac{m}{2}\pi, \ m\in\mathbb{Z}.
    \end{align*}
    If $\k_2=0$, then $\k=\k_1\in\mathbb{R}$. Thus, $\k=\overline{\k}$, which gives the desired result. If $\k_1 = \frac{m}{2}\pi$, for $m\in\mathbb{Z}$, then $\k$ satisfies \eqref{1D dispersion relation} if and only if
    \begin{align}\label{random}
        \cosh(2\k_2) = \pm A.
    \end{align}
    Since $\cosh(\cdot)$ is an even function, we have that $-\k_2$ satisfies \eqref{random} for the same frequency $\w$. Hence, $\overline{\k}=\k_1-i\k_2$ satisfies \eqref{1D dispersion relation} and this concludes the proof. 
\end{proof}

Crucially, the dispersion curves shown in Figure~\ref{fig:plot:e_cst} consist of a countable sequence of disjoint bands in which $\k$ is real valued. Between each band there is a band gap, defined in the sense of Definition~\ref{BGreal}, in which $\k$ is purely imaginary, corresponding to the decay of the wave. The occurence of $\k$ being either purely real or purely imaginary is the mechanism behind Lemma~\ref{lem:realsymm}. As we will see below, when we add singularities or damping to the model, the band gap structure is less straightforward to interpret.

\subsubsection{Singular permittivity with no damping} \label{sec:sing}

Let us now study the case where the permittivity has a dispersive (and singular) character with respect to the frequency $\omega$, but there is no damping, i.e. we consider $\a,\b>0$ and $\g=0$. This implies that
\begin{align}\label{prmtvt no damping}
    \varepsilon(\w) = \ve_0 + \frac{\alpha}{1-\beta\w^2}.
\end{align}
The interesting aspect in this setting is the existence of real poles for the permittivity. They are given by
\begin{align*}
    \w^*_{\pm} = \pm \frac{1}{\sqrt{\beta}}.
\end{align*}
In Figure \ref{fig:plot:e_no_damping}, we observe that near the pole of the permittivity there are infinitely many band-gaps. This was similarly observed recently by \cite{touboul2023dispersive}. Noting that a band gap occurs when the magnitude of the right-hand side of \eqref{1D dispersion relation} is greater than one. We define the function
\begin{align}\label{def:f}
    f(\w) := \cos\Big(\s_0(\w)\Big)\cos\Big(\r(\w)\s_0(\w)\Big) - \frac{1+\r(\w)^2}{2\r(\w)} \sin\Big(\s_0(\w)\Big) \sin\Big(\r(\w)\s_0(\w)\Big),
\end{align}
which is the right-hand side of \eqref{1D dispersion relation}. We will prove that this takes values greater than one on a countably infinite number of disjoint intervals within any neighbourhood of the singularity. To do so, we will introduce the following notation, which will be used in our analysis.

\begin{figure}
    \centering
    \includegraphics[width=0.6\linewidth]{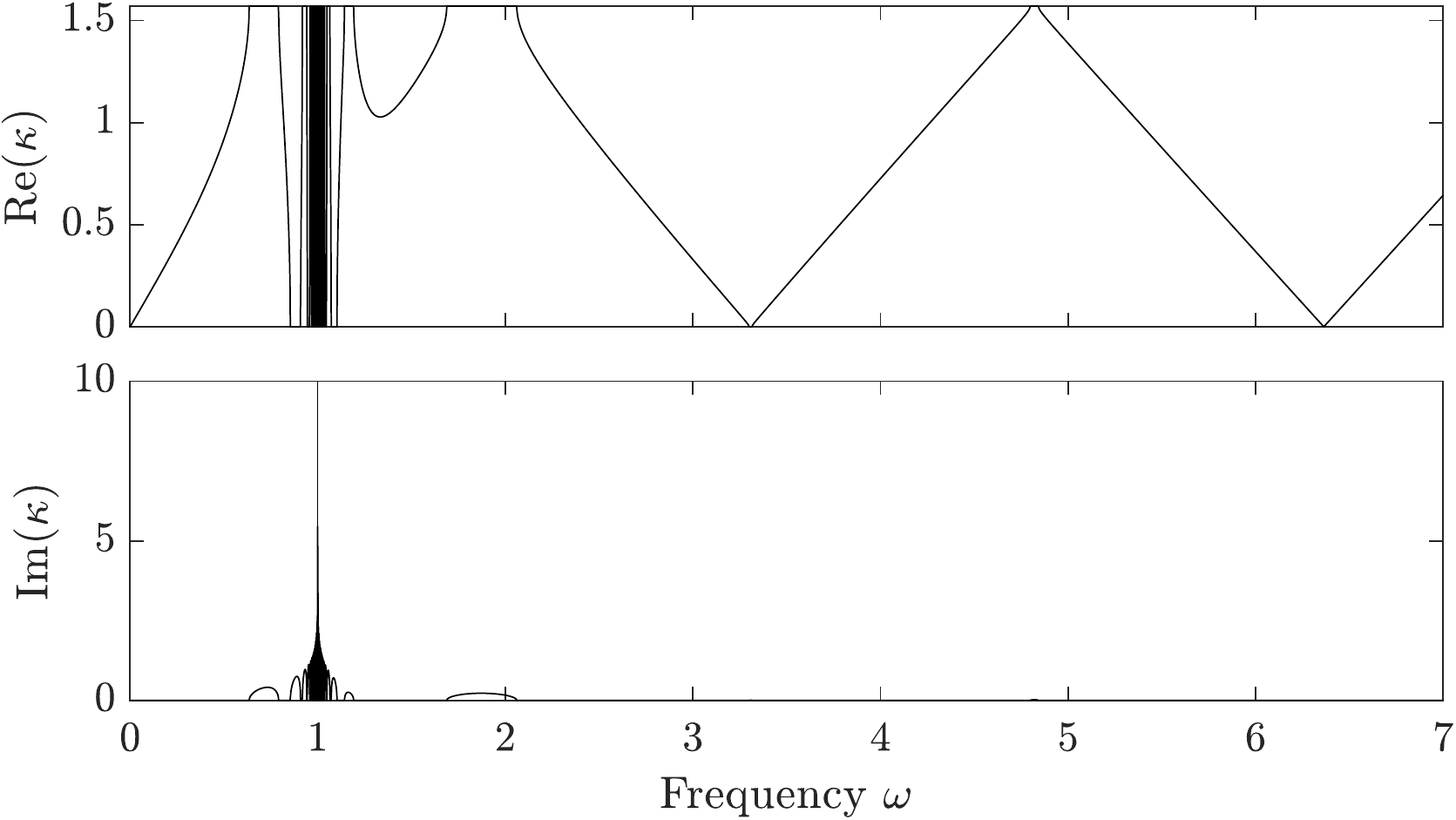}
    \begin{tikzpicture}
    \node[anchor=west] at (-1.75,1.5) {Singularities:};
    \draw[<->,\axiscol] (-1.5,0) -- (1.5,0);
    \draw[<->,\axiscol] (0,-1) -- (0,1);
    \node[anchor=north,scale=0.6,\axiscol] at (1.5,0) {Re$(\omega)$};
    \node[anchor=east,scale=0.6,\axiscol] at (0,1) {Im$(\omega)$};
    
    \node at (0.78,0) {$\times$};
    \node[anchor=south] at (0.78,0) {$\omega^*_+$};

    \node at (-0.78,0) {$\times$};
    \node[anchor=south] at (-0.78,0) {$\omega^*_-$};
    
    \draw[\axiscol] (0.78,0.1) -- (0.78,-0.1) node[at end,anchor=north,scale=0.5]{1};
    \draw[\axiscol] (-0.78,0.1) -- (-0.78,-0.1) node[at end,anchor=north,scale=0.5]{-1};
    
    \node[white] at (-2,-3) {.}; 
    \end{tikzpicture}

    \vspace{0.4cm}

    \includegraphics[width=0.48\linewidth]{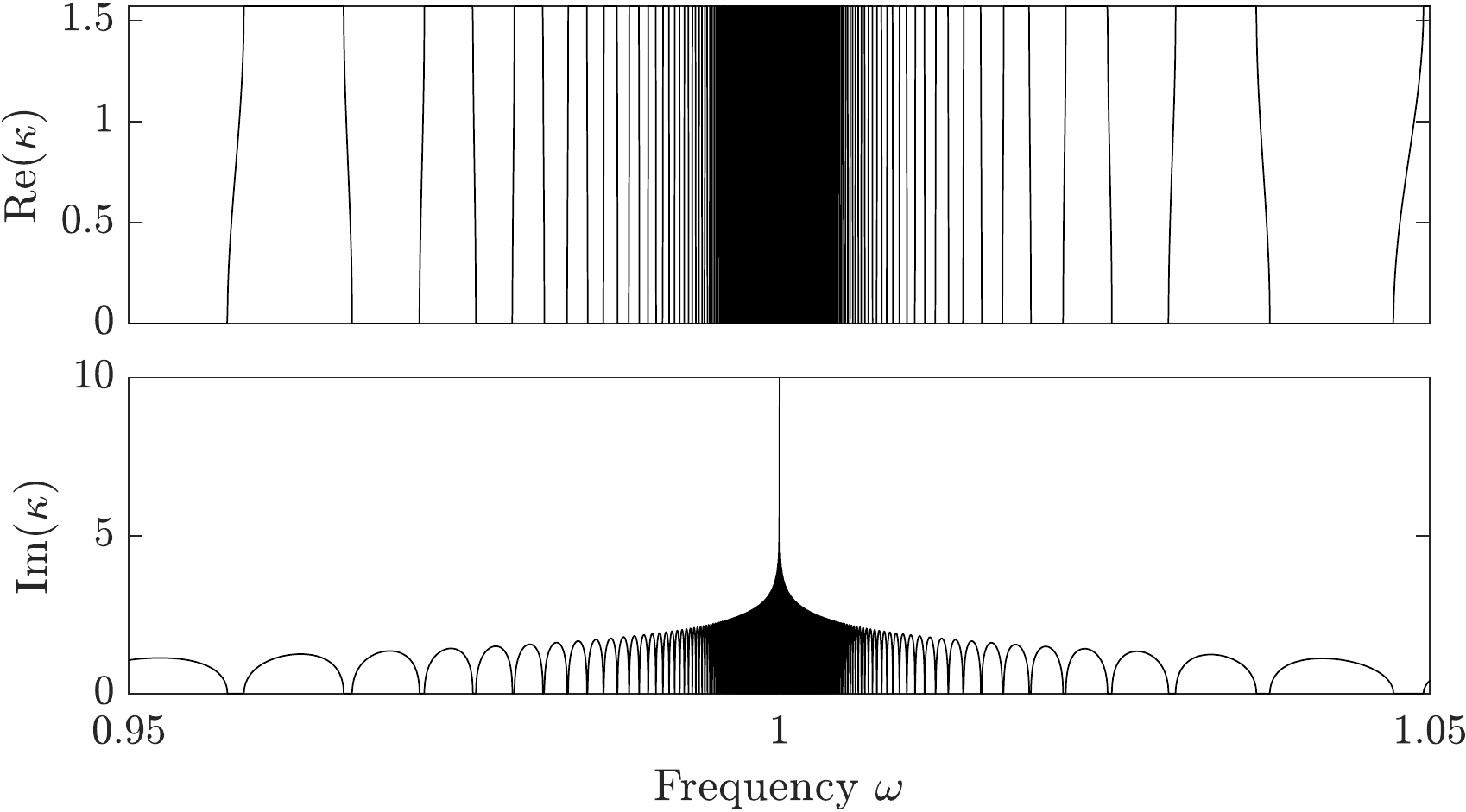}
    \includegraphics[width=0.48\linewidth]{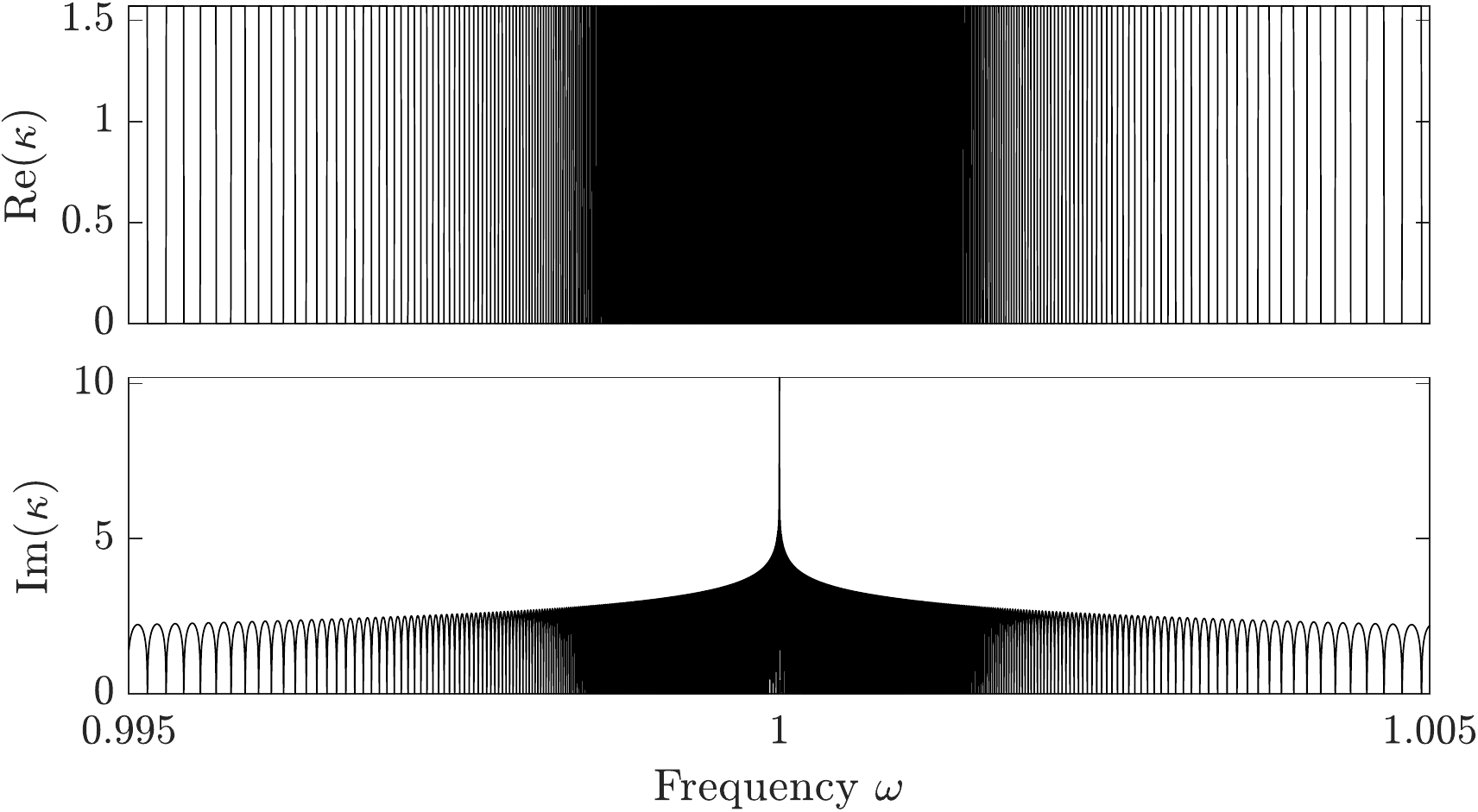}
    
    \caption{The dispersion relation of a photonic crystal with frequency-independent material parameters. We model a material with permittivity given by \eqref{prmtvt} with $\alpha=1$, $\beta=1$ and $\gamma=0$. The frequency $\omega$ is chosen to be real and the Bloch parameter $\kappa$ allowed to take complex values. The permittivity is singular when $\omega=1$. The lower two plots are display the same dispersion curves, zoomed into the region around the singularity.}
    \label{fig:plot:e_no_damping}
\end{figure}
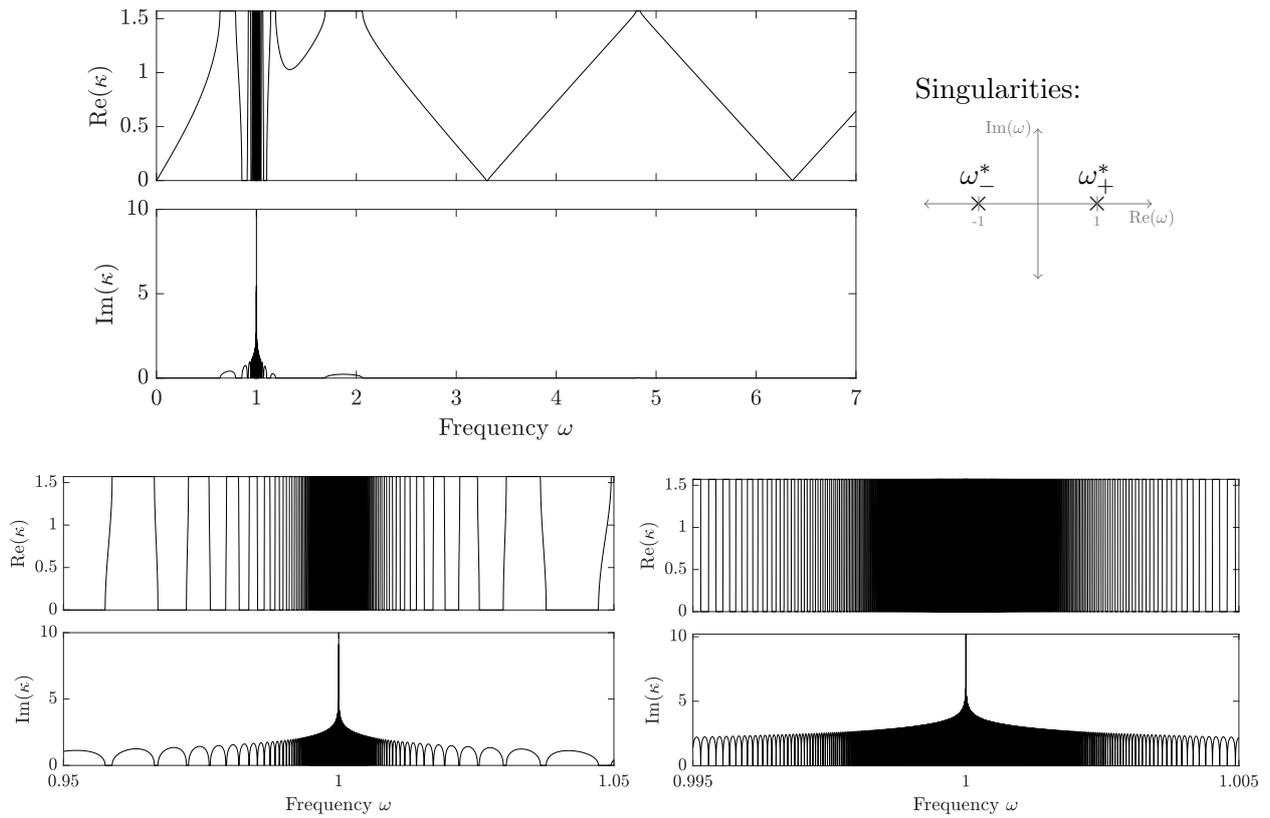

\begin{notation}
    Let $x,y\in\mathbb{R}$. Then, we use $x \downarrow y$ when $x \to y$ and $x>y$. Similarly, we use $x \uparrow y$ when $x \to y$ and $x<y$.
\end{notation}

Then, close to a permittivity pole, the following holds.

\begin{theorem}\label{thm:inf b-g}
    Let $\w^*$ denote a pole of the permittivity $\ve(\w)$ given by \eqref{prmtvt no damping}, i.e. $\w^*\in\left\{\pm\frac{1}{\sqrt{\beta}}\right\}$. Then, for $\d>0$, the intervals $[\w^*-\d,\w^*)$ and $(\w^*,\w^*+\d]$ contain infinitely many disjoint sub-intervals, denoted by $\mathcal{I}_{i}$ and $\mathcal{J}_{i}$, $i=1,2,\dots$, respectively, that are band gaps.
\end{theorem}
\begin{proof}
    We will first prove this result for the interval $[\w^*-\d,\w)$, for $\d>0$. It suffices to show that there are infinitely many points $\w^{\dagger}_i\in[\w^*-\d,\w^*)$, $i=1,2,\dots,$ for $\d>0$ for which $f(\w_i^{\dagger})>1$ or $f(\w_i^{\dagger})<-1$. Then, the continuity of $f$ around these points gives us the existence of intervals of the form $\mathcal{I}_i := [\w_i^{\dagger}-s,\w_i^{\dagger}+s]$, for $i=1,2,\dots$, for small $s>0$, such that,
    \begin{align*}
        f(\w) > 1 \ \ \text{ or } \ \ f(\w)<-1, \ \ \ \forall \ \w\in \mathcal{I}_i, \ \ i=1,2,\dots.
    \end{align*}
    From \eqref{1D dispersion relation} and \eqref{def:f}, this gives
    \begin{align*}
        \cos(2\k) > 1 \ \ \text{ or } \ \  \cos(2\k)<-1, \ \ \ \forall \ \w\in \mathcal{I}_i, \ \ i=1,2,\dots,
    \end{align*}
    This is equivalent to the $\mathcal{I}_i$'s, $i=1,2,\dots$ being band gaps, since $\k$ becomes complex in these intervals, i.e. $|\Im(\k)|\ne0$. Hence, since $|\Im(\k)|$ is continuous with respect to $\w$, we get that it has a local maximum in each of the $\mathcal{I}_i$'s, for $i=1,2,\dots$.\\
    We observe that $\lim_{\w\uparrow\w^{*}}\ve(\w) = +\infty$.  Then, this implies that $\lim_{\w\uparrow\w^{*}}\r(\w) = +\infty$, and so, we get
    \begin{align*}
        \lim_{\w\uparrow\w^{*}} \frac{1+\r(\w)^2}{2\r(\w)} = +\infty.
    \end{align*}
    Also, as $\w\uparrow\w^{*}$, we have that $\s_0$ is constant and so, without loss of generality, we can assume that $\sin(\s_0)>0$ (the same argument holds for taking $\sin(\s_0)<0$). Hence, there exists $\d_1>0$ such that for all $\w\in[\w^*-\d_1,\w^*)$, we have that
    \begin{align}
         \frac{1+\r(\w)^2}{2\r(\w)} > \frac{1}{\sin\Big(\s_0(\w)\Big)} \geq 1.
    \end{align}
    Now, since $\lim_{\w\uparrow\w^{*}}\r(\w) = +\infty$, we have that
    \begin{align}\label{s_0 r}
        \lim_{\w\uparrow\w^{*}}\r(\w)\s_0(\w) = +\infty.
    \end{align}
    This implies that, for all $K>0$, there exists $\d_2>0$ such that for all $\w\in[\w^*-\d_2,\w^*)$ it holds that $|\r(\w)\s_0(\w)|>K$.\\
    Now, let $\d := \max\{\d_1,\d_2\}$ and let $I_{\d}^{(-)} := [\w^*-\d,\w^*)$. Then, \eqref{s_0 r} implies that the exist two families of infinitely many points in $I_{\d}^{(-)}$, denoted by  $\{\w^{(+)}_i\}_{i=1,\dots,\infty}$ and $\{\w^{(-)}_i\}_{i=1,\dots,\infty}$, such that, for all $i=1,\dots,\infty$, we have
    \begin{align}
        \sin\Big( \r(\w^{(+)}_i)\w^{(+)}_i \Big) = 1 \ \ \ \text{ and } \ \ \ \sin\Big( \r(\w^{(-)}_i)\w^{(-)}_i \Big) = -1.
    \end{align}
    We also note that this implies, for all $i=1,\dots,\infty$, that
    \begin{align}
        \cos\Big( \r(\w^{(+)}_i)\w^{(+)}_i \Big) = \cos\Big( \r(\w^{(-)}_i)\w^{(-)}_i \Big) = 0.
    \end{align}
    Thus, for all $i=1,\dots,\infty$, we have
    \begin{align}
        f \Big( \w^{(+)}_i \Big) = - \frac{1+\r(\w^{(+)}_i)^2}{2\r(\w^{(+)}_i)} \sin\Big(\s_0(\w^{(+)}_i)\Big) < -1
    \end{align}
    and
    \begin{align}
        f \Big( \w^{(-)}_i \Big) = \frac{1+\r(\w^{(-)}_i)^2}{2\r(\w^{(-)}_i)} \sin\Big(\s_0(\w^{(-)}_i)\Big) > 1
    \end{align}
    In particular, without loss of generality, let us assume that 
    $\w^{(+)}_0$ is the smallest of the elements in both families $\{\w^{(+)}_i\}_{i=1,\dots,\infty}$ and $\{\w^{(-)}_i\}_{i=1,\dots,\infty}$. Then, the periodicity of $\sin(\cdot)$ shows that the elements of these families respect the following ordering:
    \begin{align}
        \w^{(+)}_0 < \w^{(-)}_0 < \w^{(+)}_1 < \w^{(-)}_1 < \dots.
    \end{align}
    Now, the continuity of $f$ around these points allows us to take $s>0$ such that
    \begin{align*}
        f(\w) <-1, \ \ \ \forall \ \w \in \Big[\w_i^{(+)}-s, \w_i^{(+)}+s\Big], \ \ \ i=1,2,\dots,
    \end{align*}
    \begin{align*}
        f(\w) >1, \ \ \ \forall \ \w \in \Big[\w_i^{(-)}-s, \w_i^{(-)}+s\Big], \ \ \ i=1,2,\dots,
    \end{align*}
    and
    \begin{align*}
        \Big[\w_i^{(+)}-s, \w_i^{(+)}+s\Big] \bigcap \Big[\w_i^{(-)}-s, \w_i^{(-)}+s\Big] = \emptyset, \ \ i=1,2,\dots \ .
    \end{align*}
    Finally, the infinity of elements in the families $\{\w^{(+)}_i\}_{i=1,\dots,\infty}$ and $\{\w^{(-)}_i\}_{i=1,\dots,\infty}$ gives us the desired result.
    
    We note that for the neighborhood of the form $\w\in(\w^*,\w^*+\d]$ for $\d>0$, the proof remains the same with the slight change of taking the limits as $\w\downarrow\w^{*}$.
\end{proof}

In addition to the occurrence of a countable number of band gaps close to the pole, in Figure \eqref{fig:plot:e_no_damping}, we observe that there is an interesting behaviour of the imaginary part $\Im(\k)$ of the quasiperiodicity $\k$ as the frequency $\w\in\mathbb{R}$ approaches a permittivity pole. In fact, we see that close to a pole, $|\Im(\k)|$ becomes arbitrarily big. This due to the resonance occurring here and is strongly related to the existence of infinitely many band gaps close to the pole. Actually, it is  a corollary of Theorem \ref{thm:inf b-g}. 

\begin{corollary}\label{lemma:f(w)}
    Let $\w\in\mathbb{R}$ and $\k\in\mathbb{C}$ be the associated quasiperiodicity satisfying the dispersion relation \eqref{1D dispersion relation} and let $f(\w)$ be the function defined in \eqref{def:f}. Let $\w^*\in\mathbb{R}$ denote a pole of the undamped permittivity $\ve(\w)$, given by \eqref{prmtvt no damping}. Then, for all $K>0$, there exists $\d>0$, such that for all $p\in[-K,K]$, there exists $\w_p\in[\w^*-\d,\w^*)$ such that $\Im(\k(\w_p)) = p$. That is, $|\Im(\k(\w))|$ takes arbitrarily large values as $\w\uparrow\w^*$. The same result holds as $\w\downarrow\w^*$.
\end{corollary}

\begin{proof}
    From Theorem \ref{thm:inf b-g}, we have that for all $K>0$, there exists $\d>0$, such that, for all $\w\in[\w^*-\d,\w)$,
    \begin{align*}
        \frac{1+\r(\w)^2}{2\r(\w)} > K.
    \end{align*}
    We have that $\cos(\s_0(\w))\cos(\r(\w)\s_0(\w))$ remains bounded close to $\w^*$ and because of the continuity of $\sin(\cdot)$, we can take $\sin(\s_0(\w))>0$ in $[\w^*-\d,\w)$. Then, it follows that, for all $K>0$, in $[\w^*-\d,\w)$,
    \begin{align*}
        \frac{1+\r(\w)^2}{2\r(\w)}\sin\Big(\s_0(\w)\Big) > K.
    \end{align*}
    But, from Theorem \eqref{thm:inf b-g}, we have the existence of infinitely many points $\w^{(+)}_0 < \w^{(-)}_0 < \w^{(+)}_1 < \w^{(-)}_1 < \dots$ in $[\w^*-\d,\w)$, for which, $\sin\Big(\r(\w)\s_0(\w)\Big)$ oscillates between 1 and -1 in each of the intervals of the form $[\w^{(+)}_0 , \w^{(-)}_0]$, $[\w^{(-)}_0 , \w^{(+)}_1]$, $\dots$, denoted by $\mathcal{I}_i$, $i=1,2,\dots.$ This implies that, for all $K>0$, for all $p\in[-K,K]$, there exists $\w_p\in\mathcal{I}_i$, for $i=1,2,\dots,$ such that $f(\w_p) = p$. Since this holds for all $K>0$, it translates to $f$ oscillating and taking all values between $+\infty$ and $-\infty$ in $[\w^*-\d,\w^*)$ as we get closer to $\w^*$.   \\
    Now, from \eqref{1D dispersion relation} and \eqref{def:f}, we get that
    \begin{align}\label{k as ln}
        \k = \frac{-i}{2}\ln\Big(f(\w)\pm\sqrt{f(\w)^2-1}\Big),
    \end{align}
    where $\ln(\cdot)$ denotes the complex logarithm. We see that
    \begin{align*}
        \Im(\k) = \frac{1}{2}\Re\Big(\ln\Big(f(\w)\pm\sqrt{f(\w)^2-1}\Big)\Big).
    \end{align*}
    Although, since we are using the complex logarithm, we have that
    \begin{align*}
        \Re\Big(\ln\Big(f(\w)\pm\sqrt{f(\w)^2-1}\Big)\Big) = \ln\Big|f(\w)\pm\sqrt{f(\w)^2-1}\Big|.
    \end{align*}
    Hence, since we have shown that in $[\w^*-\d,\w^*)$, the function $f(\w)$ oscillates between $+\infty$ and $-\infty$, we have that $\Big|f(\w)\pm\sqrt{f(\w)^2-1}\Big|$ has the same behaviour in $[\w^*-\d,\w^*)$, but the oscillation takes place between $0$ and $+\infty$. Finally, since $\ln(\cdot)$ is an increasing function, we obtain the desired result. Let us note that the proof is the same when we consider $\w\downarrow\w^*$, with the slight change that we consider neighborhoods of the form $(\w^*,\w^*+\d]$.
    
    \end{proof}


\subsubsection{Complex permittivity} \label{sec:complex}

We will now study the effect that introducing damping though allowing the permitivitty to be complex has on our one-dimensional system. Starting from the straightforward real-valued, non-dispersive model considered in Section~\ref{sec:realconst}, we subsequently add damping. For this, we take $\a\in\mathbb{C}$ and $\b=\g=0$, i.e.
\begin{align}\label{e_complex}
    \ve(\w) = \ve_0 + \a \in \mathbb{C}.
\end{align}
The dispersion curves for this setting are shown in Figure~\ref{fig:plot:e_complex}. They are plotted for $\alpha$ with successively larger imaginary parts, to show the effect of gradually increasing the damping. We see that the clear structure of successive bands and gaps is gradually blurred out, eventually to the point that the spectrum bears no clear relation to the original undamped spectrum.

In many ways, the spectrum we obtain in this setting appears to be similar to the actual halide perovskite particles, as plotted in Figure~\ref{fig:HPdispersion}. Indeed, all of the results proved in subsection \ref{Properties of the dispersion relation} hold, with the exception of the imaginary part decay. In fact, the converse is true, as made precise by the following result.

\begin{figure} 
    \centering
    \includegraphics[width=0.55\linewidth]{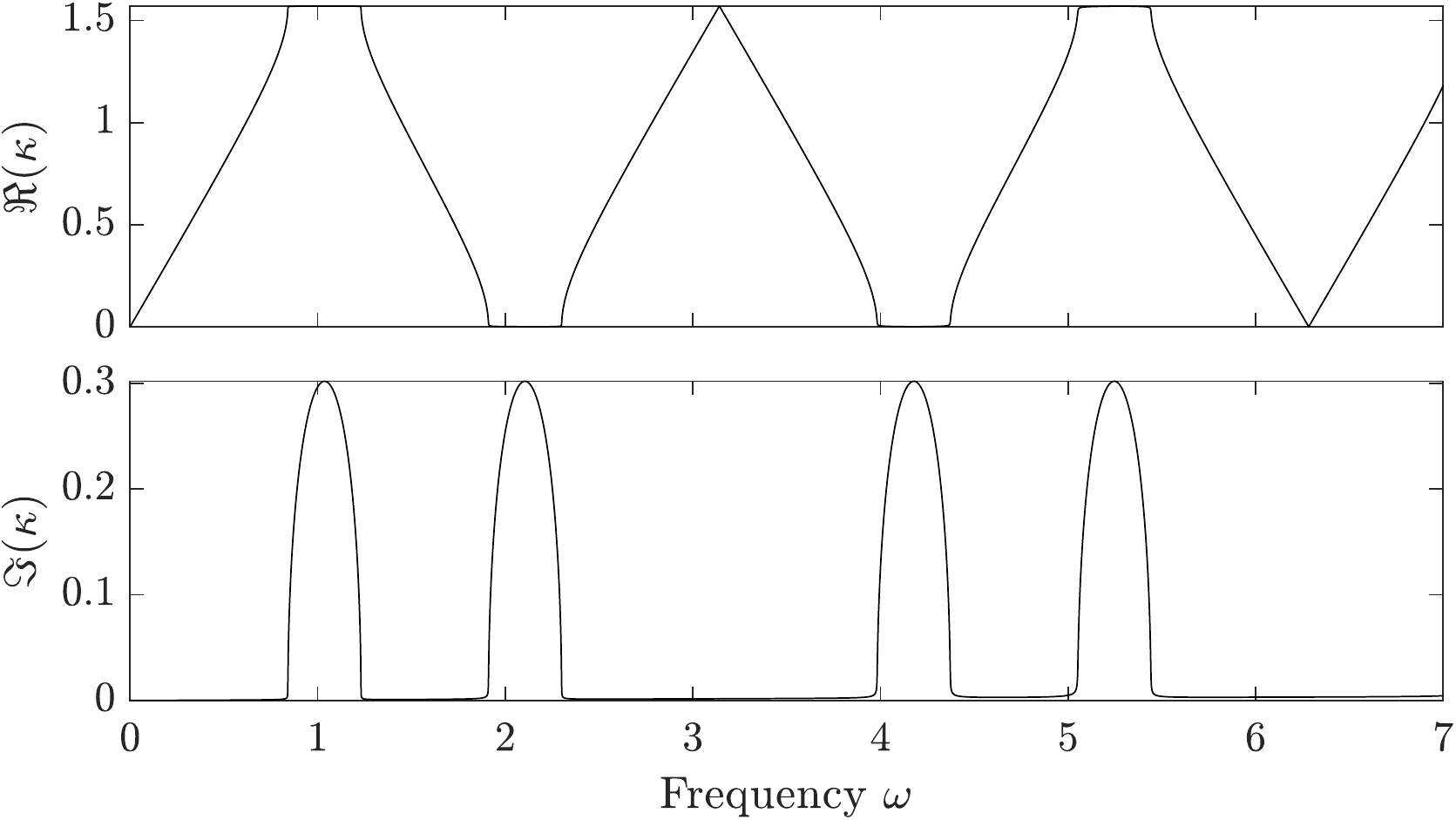}
    \begin{tikzpicture}
    \node[anchor=west] at (-1.75,1.5) {(a) $\alpha = 1+0.001i$};
    \node[white] at (-2,-3) {.};
    \node[white] at (1.3,-3) {.};
    \end{tikzpicture}

    \vspace{0.4cm}
    
    \includegraphics[width=0.55\linewidth]{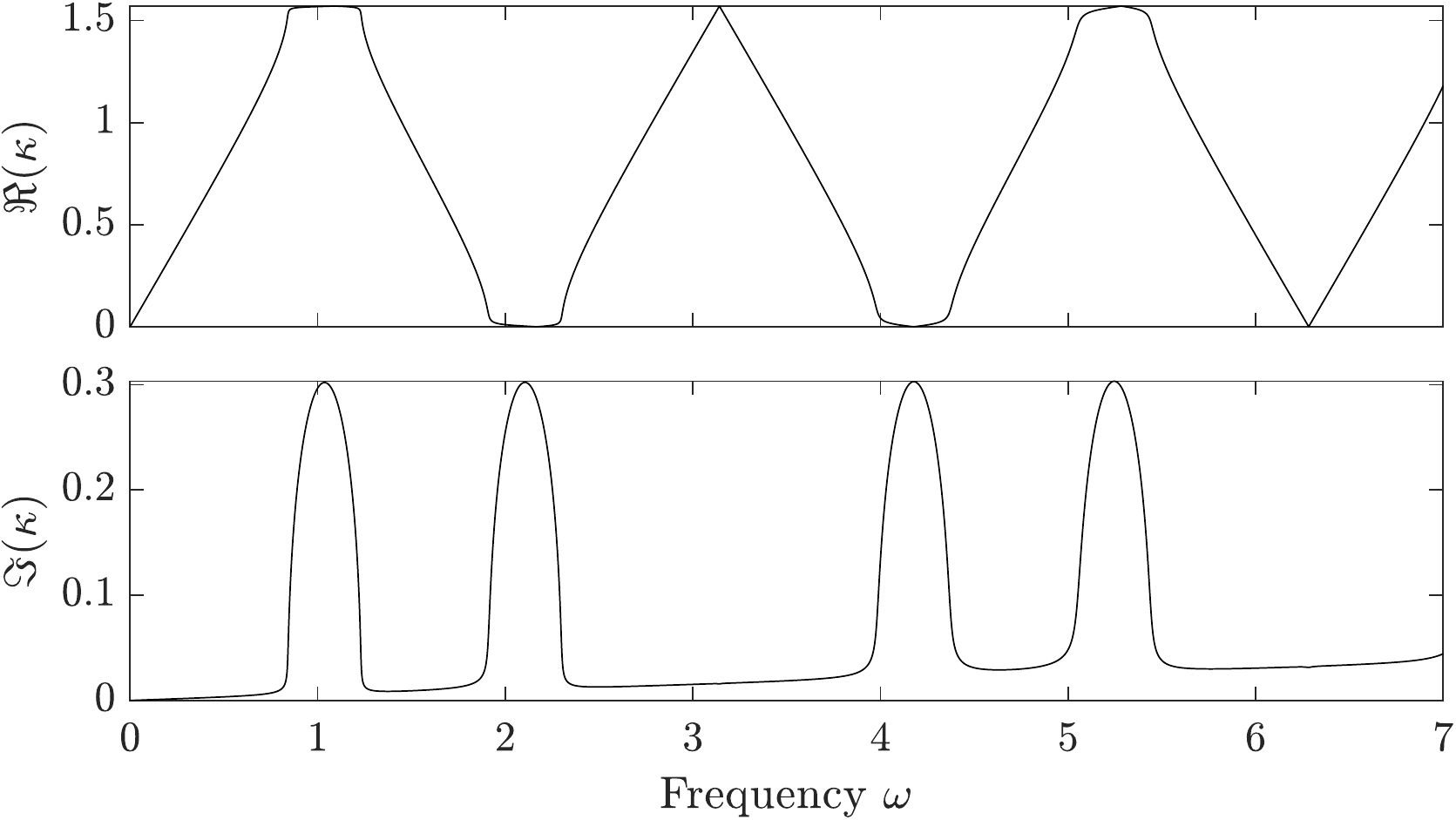}
    \begin{tikzpicture}
    \node[anchor=west] at (-1.75,1.5) {(b) $\alpha = 1+0.01i$};
    \node[white] at (-2,-3) {.};
    \node[white] at (1.3,-3) {.};
    \end{tikzpicture}

    \vspace{0.4cm}
    
    \includegraphics[width=0.55\linewidth]{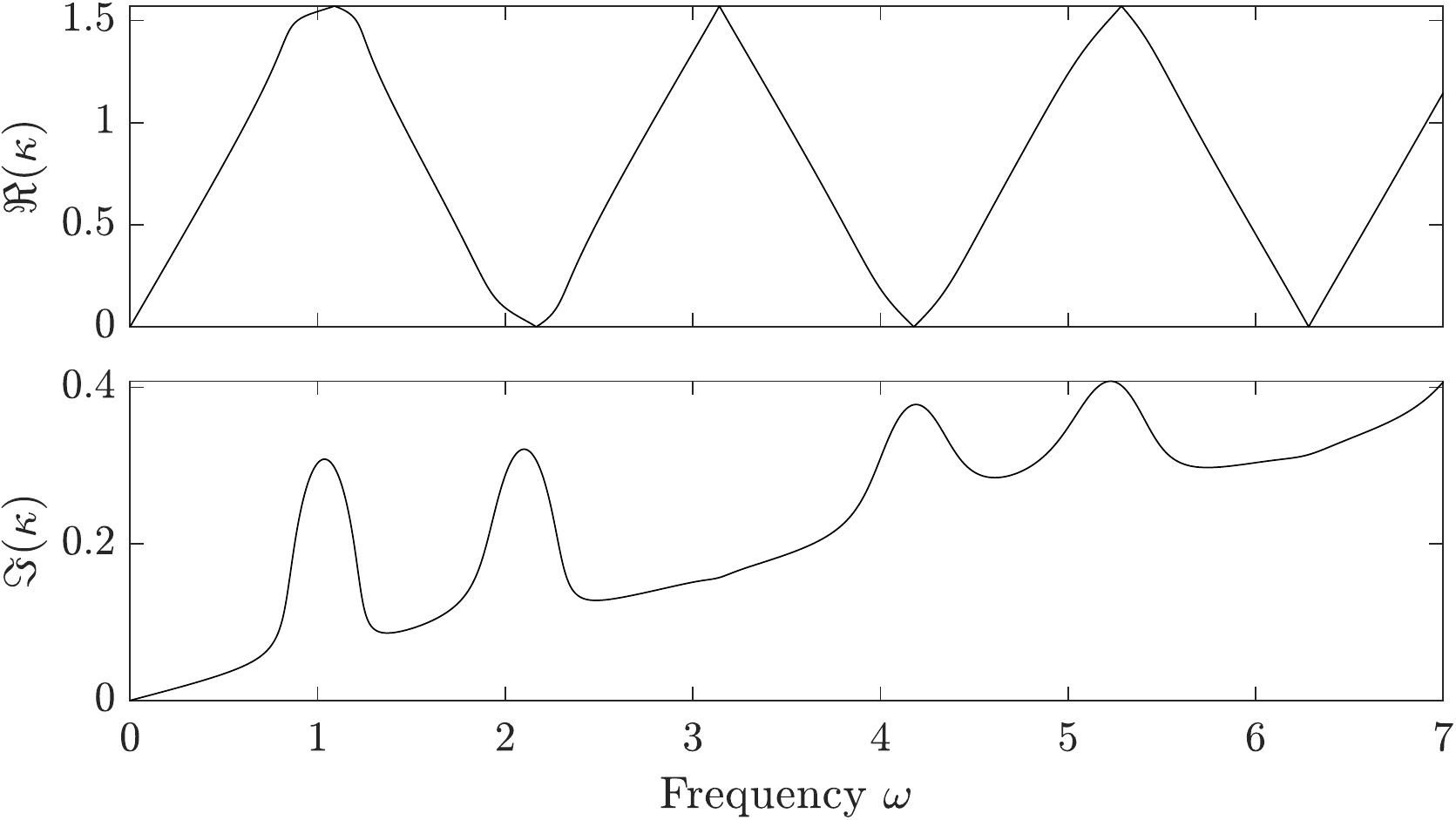}
    \begin{tikzpicture}
    \node[anchor=west] at (-1.75,1.5) {(c) $\alpha = 1+0.1i$};
    \node[white] at (-2,-3) {.};
    \node[white] at (1.3,-3) {.};
    \end{tikzpicture}

    \vspace{0.4cm}

    \includegraphics[width=0.55\linewidth]{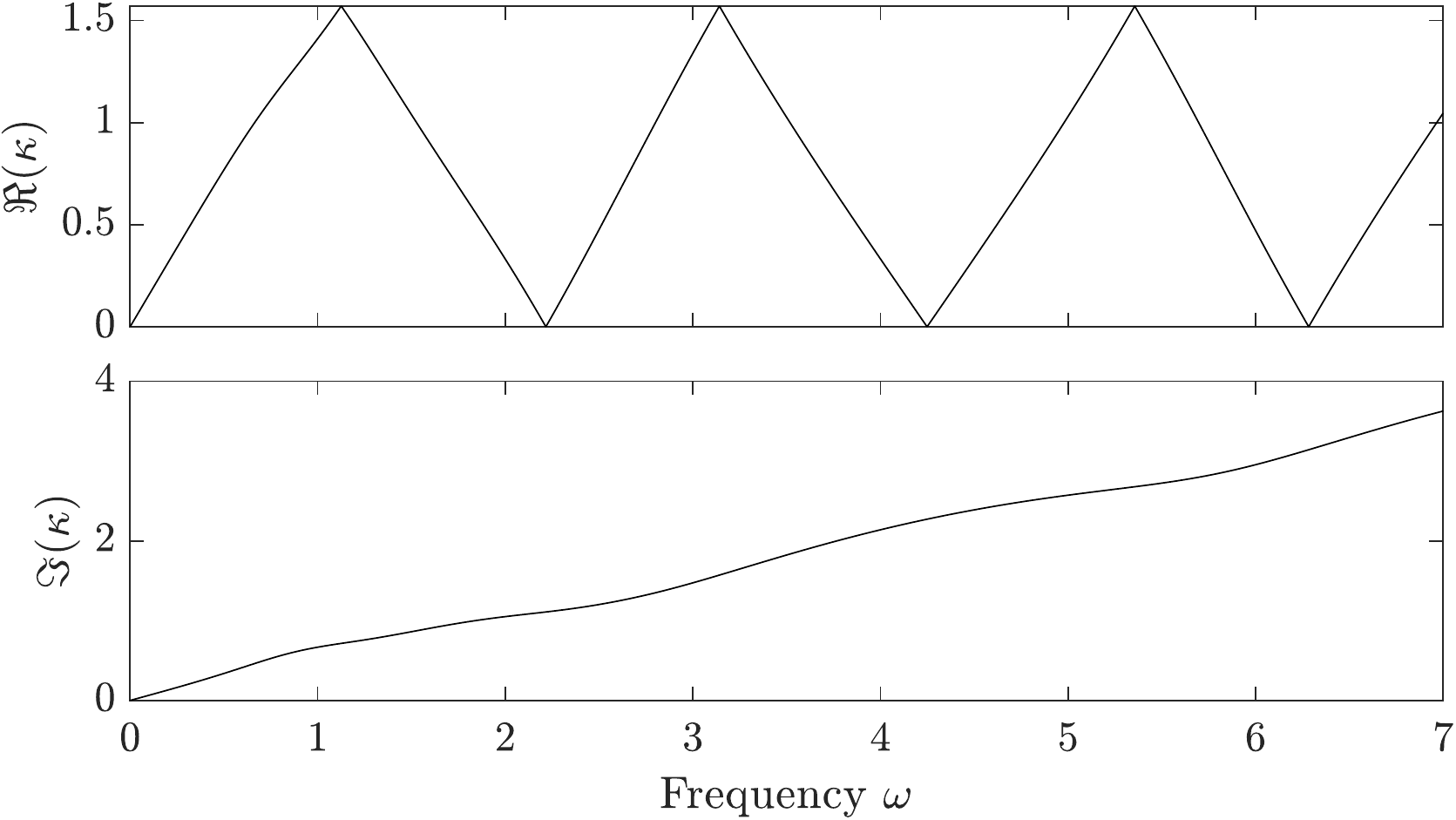}
    \begin{tikzpicture}
    \node[anchor=west] at (-1.75,1.5) {(d) $\alpha = 1+i$};
    \node[white] at (-2,-3) {.};
    \node[white] at (1.3,-3) {.};
    \end{tikzpicture}
    \caption{The dispersion relation of a photonic crystal with frequency-independent complex-valued material parameters. We model a material with permittivity given by \eqref{prmtvt} with $\alpha\in\mathbb{C}$, $\beta=0$ and $\gamma=0$. The frequency $\omega$ is chosen to be real and the Bloch parameter $\kappa$ allowed to take complex values. The permittivity is never singular in this case. }
    \label{fig:plot:e_complex}
\end{figure}

\begin{lemma}
    Let $\k\in\mathbb{C}$ and $\omega\in\mathbb{R}$ be a quasiperiodicity and a frequency, respectively, satisfying the dispersion relation \eqref{1D dispersion relation} with complex-valued, non-dispersive permittivity given by \eqref{e_complex}. Then, 
    \begin{align}
        \lim_{\w\to+\infty} |\Im(\k)| = +\infty.
    \end{align}
\end{lemma}
\begin{proof}
    Let us recall that $\k_2$, denoting the imaginary part of $\k\in\mathbb{C}$, is given by \eqref{a_2}, where $\mathscr{L}_1$ and $\mathscr{L}_2$ are given by \eqref{L1} and \eqref{L2}, respectively. We have that
    \begin{align}\label{limit r s_0}
        \lim_{\w\to\infty}\r(\w)\s_0(\w) = +\infty.
    \end{align}
    since $\s_0$ is linear with respect to $\w$ and $\r$ does not depend on $\w$ in this setting. This implies
    \begin{align*}
    \lim_{\w\to\infty}\sinh\Big(\r(\w)\s_0(\w)\Big) = +\infty \ \ \text{ and } \ \ \lim_{\w\to\infty}\cosh\Big(\r(\w)\s_0(\w)\Big) = +\infty.
    \end{align*}
    We note that it is enough to show that
    \begin{align*}
        \lim_{\w\to\infty} |\mathscr{L}_1| = \lim_{\w\to\infty} |\mathscr{L}_2| = +\infty,
    \end{align*}
    since applying this on \eqref{a_2} gives that $|\k_2|\to+\infty$ as $\w\to+\infty$. Indeed, \eqref{L1} gives that 
    \begin{align*}
        \mathscr{L}_1 = C_1 \cosh(\s_0\r_2) - C_2 \sinh(\s_0\r_2),
    \end{align*}
    where
    \begin{align*}
        C_1 := \cos(\s_0) \cos(\s_0\r_1) - \frac{\r_1(1+\r_1^2+\r_2^2)}{2(\r_1^2+\r_2^2)} \sin(\s_0) \sin(\s_0\r_1)
    \end{align*}
    and
    \begin{align*}
        C_2 := \frac{\r_2(\r_1^2-1+\r_2^2)}{2(\r_1^2+\r_2^2)} \sin(\s_0) \cos(\s_0\r_1).
    \end{align*}
    Using the exponential formulation of the hyperbolic trigonometric functions, we get 
    \begin{align*}
        \mathscr{L}_1 = \frac{C_1-C_2}{2}e^{\s_0 \r_2} + \frac{C_1+C_2}{2}e^{-\s_0\r_2}.
    \end{align*}
    Now, we observe that, as $\w\to+\infty$, $C_1$ and $C_2$ are both bounded and $C_1-C_2\ne0$. Then, from \eqref{limit r s_0}, we get that
    \begin{align*}
        \lim_{\w\to\infty} |\mathscr{L}_1| = + \infty.
    \end{align*}
    Similarly, \eqref{L2} gives that 
    \begin{align*}
        \mathscr{L}_2 = \tilde{C}_1 \sinh(\s_0\r_2) - \tilde{C}_2 \cosh(\s_0\r_2),
    \end{align*}
    where
    \begin{align*}
        \tilde{C}_1 := \cos(\s_0) \sin(\s_0\r_1) + \frac{\r_1(1+\r_1^2+\r_2^2)}{2(\r_1^2+\r_2^2)} \sin(\s_0) \cos(\s_0 \r_1)
    \end{align*}
    and
    \begin{align*}
        \tilde{C}_2 := \frac{\r_2(\r_2^2-1+\r_1^2)}{2(\r_1^2+\r_2^2)} \sin(\s_0) \sin(\s_0 \r_1)
    \end{align*}
    Then, we can write
    \begin{align*}
        \mathscr{L}_2 = \frac{\tilde{C}_1 + \tilde{C}_2}{2} e^{\s_0\r_2} + \frac{\tilde{C}_2 - \tilde{C}_1}{2} e^{-\s_0\r_2}
    \end{align*}
    As before, we observe that, as $\w\to+\infty$, $\tilde{C}_1$ and $\tilde{C}_2$ are both bounded and $\tilde{C}_1-\tilde{C}_2\ne0$. Then, from \eqref{limit r s_0}, we get that
    \begin{align*}
        \lim_{\w\to\infty} |\mathscr{L}_2| = + \infty.
    \end{align*}
    This concludes the proof
\end{proof}

\subsubsection{Discussion}

The analysis in this section can be used to understand the dispersion diagram for the halide perovskite photonic crystal that was presented in Figure~\ref{fig:HPdispersion}. There are two crucial observations. First, we saw in Section~\ref{sec:sing} that the introduction of singularities in the permittivity led to the creation of countably infinitely many band gaps in a neigbourhood of the pole, when the pole falls on the real axis. However, this exotic behaviour is not seen in Figure~\ref{fig:HPdispersion}, due in part to the introduction of damping causing the poles to fall below the real axis. This effect can also be exaplined in terms of the results in Section~\ref{sec:complex}, where we saw that the introduction of damping to a simple non-dispersive model smoothed out the band gaps. The behaviour shown in Figure~\ref{fig:HPdispersion} is a combination of these phenomena.




\section{Multiple dimensions} \label{sec:multid}

Let us now treat the periodic structures in two- and three-dimensions. In this case, to be able to handle the problem concisely using asymptotic methods, we are interested in the case of small resonators. We will assume that there exists some fixed domain $D$, which the the union of the $N$ disjoint subsets $D = D_1 \cup D_2 \cup \dots\cup D_N$, such that $\W$ is given by
\begin{equation}\label{Omega}
    \W = \d D + z,
\end{equation}
for some position $z\in\mathbb{R}^d$, $d=2,3$, and characteristic size $0<\d\ll1$. Then, making a change of variables, the quasiperiodic Helmholtz problem (\ref{Quasiperiodic Helmholtz problem}) becomes
\begin{align}\label{Quasiperiodic Helmholtz problem 2 }
    \begin{cases}
    \D u^{\k} + \d^2 \w^2 \ve(\w)\m_0 u^{\k} = 0 \ \ \ &\text{ in } \mathcal{D}, \\
    \D u^{\k} + \d^2 k_0^2 u^{\k} = 0 &\text{ in } \mathbb{R}^d\setminus\overline{D}, \\
    u^{\k}|_+ - u^{\k}|_-=0 &\text{ on } \partial\mathcal{D}, \\
    \frac{\partial u^{\k}}{\partial \n}|_+ - \frac{\partial u^{\k}}{\partial \n}|_- = 0 &\text{ on } \partial\mathcal{D}, \\
    u^{\k}(x_d,x_0) &\text{ is $\k$-quasiperiodic in $x_d$}, \\
    u^{\k}(x_d,x_0) &\text{ satisfies the $\k$-quasiperiodic radiation condition as } |x_0|\to\infty.\\
    \end{cases}
\end{align}
We will also make an additional assumption on the dimensions of the nano-particles. This will allow us to prove an approximation for the values of the modes $u|_{D_i}$, $i=1,\dots,N,$ on each particle. The assumption is one of \emph{diluteness}, in the sense that the particles are small relative to the separation distances between them. To capture this, we introduce the parameter $\rho_i$ to capture the size of the reference particles $D_1,\dots,D_N$. We define $\rho_i:=\tfrac{1}{2}(\text{diam}(D_i))$ where $\text{diam}(D_i)$ is given by
\begin{equation}
    \text{diam}(D_i)=\sup\{  |x-y|:x,y\in D_i  \}.
\end{equation} 
We will assume that each $\rho_i\to0$ independently of $\delta$. This regime means that the system is dilute in the sense that the particles are small relative to the distances between them.

We will first present certain general results for this system. Then we will give a more qualitative description of the two-dimensional setting.

\subsection{Integral formulation}

Let $G^k(x)$ be the outgoing Helmholtz Green's function in $\mathbb{R}^d$, defined as the unique solution to
\begin{align*}
    (\D + k^2) G^k(x) = \d_0(x) \ \ \text{ in } \mathbb{R}^d,
\end{align*}
along with the outgoing radiation condition, where $\delta_0$ is the Dirac delta. It is well known that $G^k$ is given by
\begin{align}\label{G}
    G^k(x)=
    \begin{cases}
    -\frac{i}{4}H^{(1)}_0(k|x|), \ \ \ &d=2,\\
    -\frac{e^{ik|x|}}{4\pi|x|}, &d=3,\\
    \end{cases}
\end{align}
where $H_0^{(1)}$ is the Hankel function of first kind and order zero. We define the quasiperiodic Green's function $G^{\k,k}(x)$ as the Floquet transform of $G^k(x)$ in the first $d_l$ coordinate dimensions, i.e.
\begin{align}\label{quasi G}
    G^{\k,k}(x)=
    \begin{cases}
    -\frac{i}{4}\sum\limits_{m\in\L}H^{(1)}_0(k|x-m|)e^{im\cdot\k}, \ \ \ &d=2,\\
    - \frac{1}{4\pi} \sum\limits_{m\in\L} \frac{e^{ik|x-m|}}{|x-m|}e^{i\k\cdot m}, &d=3.\\
    \end{cases}
\end{align}
Then, from \cite{AD,alexopoulos2022mathematical}, we know that \eqref{Quasiperiodic Helmholtz problem} has the following integral representation expression.

\begin{theorem}[Lippmann-Schwinger integral representation formula] \label{Integral representation}
A function $u^{\k}$ satisfies the differential system \eqref{Quasiperiodic Helmholtz problem} if and only if it satisfies the following equation
\begin{align}\label{Lipmann-Schwinger}
    u^{\k}(x)-u^{\k}_{in}(x) = -\d^2 \w^2 \xi(\w) \int_{\mathcal{D}} G^{\k,\d k_0}(x-y) u(y) \upd y, \ \ x\in\mathbb{R}^d,
\end{align}
where the function $\xi:\mathbb{C}\to\mathbb{C}$ describes the permittivity contrast between $\mathcal{D}$ and the background and is given by
$$
\xi(\w) = \m_0(\ve(\w) - \ve_0).
$$
\end{theorem}

\subsection{Dispersion relation}

We will now retrieve an expression which relates the subwavelength resonances of the system and the quasiperiodicities. The method used is similar to the one developed in \cite{alexopoulos2022mathematical} for systems of finitely many particles.

\subsubsection{Matrix representaion}

Let us define the following $\k$-quasiperiodic integral operators
\begin{align}\label{K_Di}
    K^{\k,r}_{D_i}: u|_{D_i} \in L^2(D_i) \rightarrow - \int_{D_i} G^{\k,r}(x-y)u(y)\upd y \Big|_{D_i} \in L^2(D_i)
\end{align}
and
\begin{align}\label{R_DiDj}
    R^{\k,r}_{D_i D_j}: u|_{D_i} \in L^2(D_i) \rightarrow - \int_{D_i} G^{\k,r}(x-y)u(y)\upd y \Big|_{D_j} \in L^2(D_j).
\end{align}
Then, as in \cite{alexopoulos2022mathematical}, the scattering problem has the following matrix representation:
\begin{align*}
    \begin{pmatrix}
    1-\d^2\w^2\xi(\w)K^{\k,\d k_0}_{D_1} & -\d^2\w^2\xi(\w) R^{\k,\d k_0}_{D_2 D_1} & \dots & -\d^2\w^2\xi(\w) R^{\k,\d k_0}_{D_N D_1} \\
    -\d^2\w^2\xi(\w) R^{\k,\d k_0}_{D_1 D_2} & 1-\d^2\w^2\xi(\w)K^{\k,\d k_0}_{D_2} & \dots & -\d^2\w^2\xi(\w) R^{\k,\d k_0}_{D_N D_2}\\
    \vdots & \vdots & \ddots & \vdots \\
    -\d^2\w^2\xi(\w) R^{\k,\d k_0}_{D_1 D_N} & -\d^2\w^2\xi(\w) R^{\k,\d k_0}_{D_2 D_N} & \dots & 1-\d^2\w^2\xi(\w)K^{\k,\d k_0}_{D_N}
    \end{pmatrix}
    \begin{pmatrix}
    u^{\k}|_{D_1} \\ u^{\k}|_{D_2} \\ \vdots \\ u^{\k}|_{D_N}
    \end{pmatrix}
    =
    \begin{pmatrix}
    u^{\k}_{in}|_{D_1} \\ u^{\k}_{in}|_{D_2} \\ \vdots \\ u^{\k}_{in}|_{D_N}
    \end{pmatrix}
\end{align*}
Since the scattered field is fully determined by the value within each resonator, we will introduce the notation
\begin{align}
    u_i^{\k} := u^{\k}|_{D_i}, \ \ i=1,\dots,N.
\end{align}
Then, the resonance problem is to find $\w\in\mathbb{C}$, such that there exists $(u^{\k}_1,u^{\k}_2,\dots,u^{\k}_N) \in L^2(D_1) \times L^2(D_2) \times \dots \times L^2(D_N)$, $u^{\k}_i \ne 0$, for $i=1,\dots,N,$ such that
\begin{align}\label{Resonance pb}
    \begin{pmatrix}
    1-\d^2\w^2\xi(\w)K^{\k,\d k_0}_{D_1} & -\d^2\w^2\xi(\w) R^{\k,\d k_0}_{D_2 D_1} & \dots & -\d^2\w^2\xi(\w) R^{\k,\d k_0}_{D_N D_1} \\
    -\d^2\w^2\xi(\w) R^{\k,\d k_0}_{D_1 D_2} & 1-\d^2\w^2\xi(\w)K^{\k,\d k_0}_{D_2} & \dots & -\d^2\w^2\xi(\w) R^{\k,\d k_0}_{D_N D_2}\\
    \vdots & \vdots & \ddots & \vdots \\
    -\d^2\w^2\xi(\w) R^{\k,\d k_0}_{D_1 D_N} & -\d^2\w^2\xi(\w) R^{\k,\d k_0}_{D_2 D_N} & \dots & 1-\d^2\w^2\xi(\w)K^{\k,\d k_0}_{D_N}
    \end{pmatrix}
    \begin{pmatrix}
    u^{\k}_1 \\ u^{\k}_2 \\ \vdots \\ u^{\k}_N
    \end{pmatrix}
    =
    \begin{pmatrix}
    0 \\ 0 \\ \vdots \\ 0
    \end{pmatrix}
\end{align}

\subsubsection{Resonances}

Let us now retrieve the relation between the subwavelength resonant frequencies and the quasiperiodicities, obtained by studying the solutions to \eqref{Resonance pb}. We will first recall a definition and a lemma which will help in the analysis of the problem.

\begin{definition}
Given $N\in\mathbb{N}$, we denote by $\lfloor N \rfloor:\mathbb{N}\to\{1,2,\dots,N\}$ a modified version of the modulo function, \emph{i.e.} the remainder of euclidean division by $N$. In particular, for all $M \in \mathbb{N}$, there exists unique $\t\in\mathbb{Z}^{\geq0}$ and $r\in\mathbb{N}$ with $0<r\leq N$, such that
$$
M = \t \cdot N + r.
$$
Then, we define $M \lfloor N \rfloor$ to be
$$
M \lfloor N \rfloor := r.
$$
\end{definition}

We recall the diluteness assumption that we have made on our system, which is captured by considering small particle size $\r$. We define $\rho:=\tfrac{1}{2}\max_{i} (\text{diam}(D_i))$ where $\text{diam}(D_i)$ is given by
\begin{equation}
    \text{diam}(D_i)=\sup\{  |x-y|:x,y\in D_i  \}.
\end{equation}
The next lemma is a variation of Lemma 2.6 in \cite{alexopoulos2022mathematical}.

\begin{lemma}\label{approx}
For all $i=1,\dots,N$, we denote $u^{\k}_i=u^{\k}|_{D_i}$, where $u^{\k}$ is a resonant mode, in the sense that it is a solution to \eqref{Lipmann-Schwinger} with no incoming wave. Then, for characteristic size $\d$ of the same order as $\rho$, we can write that
\begin{align}\label{approx formula}
    u^{\k}_i = \langle u^{\k},\phi^{(i)}_{\k} \rangle \phi^{(i)}_{\k} + O(\r^2), \ \ i=1,\dots,N,
\end{align}
as $\r\to0$, where $\phi^{(i)}_{\k}$ denotes the eigenvector associated to the particle $D_i$ of the potential $K_{D_i}^{\k,\d k_0}$ and $\r>0$ denotes the particle size parameter of $D_1,\dots,D_N$. Here, $\d$ and $\r$ are of the same order in the sense that $\d=O(\r)$ and $\r=O(\d)$. In this case, the error term holds uniformly for any small $\d$ and $\r$ in a neighbourhood of 0. 
\end{lemma}

Let us now state the main result of this section.

\begin{theorem}\label{determinant}
    The resonance problem, as $\d\to0$ and $\r\to0$, with $\d=O(\r)$ and $\r=O(\d)$, \eqref{Resonance pb} in dimensions $d=2,3$, becomes finding $\w\in\mathbb{C}$ such that
    \begin{align*}
        \det\Big(\mathcal{K}^{\k}(\w)\Big) = 0,
    \end{align*}
    where
    \begin{align}\label{K}
        \mathcal{K}^{\k}(\w)_{ij} := 
        \begin{cases}
            \langle R^{\k,\d k_0}_{D_i D_{i+1 \lfloor N \rfloor }} \phi^{(i)}_{\k}, \phi_{\k}^{(i+1\lfloor N \rfloor)} \rangle , &\text{ if $i=j$,} \\
            - \mathscr{A}^{\k}_i(\w,\d) \langle R^{\k,\d k_0}_{D_j D_i} \phi_{\k}^{(j)}, \phi^{(i)}_{\k} \rangle \langle R^{\k,\d k_0}_{D_i D_{i+1 \lfloor N \rfloor }} \phi_{\k}^{(i)}, \phi_{\k}^{(i+1\lfloor N \rfloor)} \rangle, &\text{ if $i \ne j$,}
        \end{cases}
    \end{align}
    Here, $k_0=\omega\sqrt{\mu_0\varepsilon_0}$ and
    \begin{align}\label{A_i}
        \mathscr{A}^{\k}_i(\w,\d) := \frac{\d^2 \w^2 \xi(\w)}{1 - \d^2 \w^2 \xi(\w) \l^{(i)}_{\k}}, \quad i=1,...,N.
    \end{align}
    with $\l_{\k}^{(i)}$ and $\phi^{(i)}_{\k}$ being the eigenvalues and the respective eigenvectors associated to the particle $D_i$ of the potential $K^{\k,\d k_0}_{D_i}$, for $i=1,2,\dots,N$.
\end{theorem}

\begin{proof}
    We will provide an outline of the proof of this result, since it follows the exact same reasoning as, for example, the proof of Theorem 2.8 in \cite{alexopoulos2022mathematical}. Using the following pole pencil decomposition,
    \begin{align}\label{pole pencil}
        \Big(1 - \d^2 \w^2 \xi(\w) K^{\k,\d k_0}_{D_i}\Big)^{-1}(\cdot) = \frac{ \langle \cdot, \phi_{\k}^{(i)} \rangle \phi_{\k}^{(i)} }{1 - \d^2 \w^2 \xi(\w) \l^{(i)}_{\k}} + R_i[\w](\cdot), \quad i=1,\dots,N,
    \end{align}
    we get that \eqref{Resonance pb} is equivalent to the system of equations
    \begin{align*}
        u^{\k}_i - \frac{\d^2 \w^2 \xi(\w)}{1 - \d^2 \w^2 \xi(\w) \l^{(i)}_{\k}} \displaystyle\sum\limits_{j=1,j \ne 1}^N \langle R^{\k,\d k_0}_{D_j D_i} u^{\k}_j, \phi_{\k}^{(i)}  \rangle \phi_{\k}^{(i)}  &= 0, \quad \text{ for each } i=1,\dots,N.
    \end{align*}
    The above system is equivalent to
    \begin{align*}
        \langle R^{\k,\d k_0}_{D_i D_{i+1 \lfloor N \rfloor}} u^{\k}_i, \phi^{(i+1 \lfloor N \rfloor)}_{\k}  \rangle - \frac{\d^2 \w^2 \xi(\w)}{1 - \d^2 \w^2 \xi(\w) \l^{(i)}_{\k}} \sum\limits_{j=1,j \ne i}^N \langle R^{\k,\d k_0}_{D_j D_i} u^{\k}_j, \phi^{(i)}_{\k}  \rangle \langle R^{\k,\d k_0}_{D_i D_{i+1 \lfloor N \rfloor}} \phi^{(i)}_{\k}, \phi^{(i+1 \lfloor N \rfloor)}_{\k}  \rangle  = 0.
    \end{align*}
    From Lemma \ref{approx}, we have
    \begin{align*}
        u^{\k}_i \simeq \langle u^{\k}, \phi_{\k}^{(i)} \rangle \phi_{\k}^{(i)}, 
    \end{align*}
    which gives
    \begin{align} 
    \mathcal{K}^{\k}(\w)
    \begin{pmatrix}
        \langle u^{\k}, \phi^{(1)}_{\k} \rangle \\
        \langle u^{\k}, \phi^{(2)}_{\k} \rangle \\
        \vdots \\
        \langle u^{\k}, \phi^{(N)}_{\k} \rangle
    \end{pmatrix}
    =
    \begin{pmatrix}
        0 \\
        0 \\
        \vdots \\
        0
    \end{pmatrix},
    \end{align}
    where
    \begin{align}\label{K2}
        \mathcal{K}^{\k}(\w)_{ij} := 
        \begin{cases}
            \langle R^{\k,\d k_0}_{D_i D_{i+1 \lfloor N \rfloor }} \phi^{(i)}_{\k}, \phi_{\k}^{(i+1\lfloor N \rfloor)} \rangle , &\text{ if $i=j$,} \\
            - \mathcal{A}^{\k}_i(\w,\d) \langle R^{\k,\d k_0}_{D_j D_i} \phi_{\k}^{(j)}, \phi^{(i)}_{\k} \rangle \langle R^{\k,\d k_0}_{D_i D_{i+1 \lfloor N \rfloor }} \phi_{\k}^{(i)}, \phi_{\k}^{(i+1\lfloor N \rfloor)} \rangle, &\text{ if $i \ne j$.}
        \end{cases}
    \end{align}
    Then, for the system to have a non-trivial solution, we require
    \begin{align*}
        \det\Big( \mathcal{K}^{\k}(\w) \Big) = 0,
    \end{align*}
    which gives the desired result.
\end{proof}

\subsection{The two-dimensional case}

In the particular case of a two-dimensional system, it is possible to provide a more detailed and simplified version of the result in Theorem \ref{determinant}. In particular, we will consider the setting where the periodic structure $\mathcal{D}$ is composed of $N\in\mathbb{N}$ resonators, denoted by $\mathcal{D}_i$, $i=1,\dots,N$, which are repeated periodically in the lattice $\L$. Let us recall that in the dimension $d=2$, the $\k$-quasiperiodic Green's function $G^{\k,r}(x)$ is given by
\begin{align}\label{quasi G 2D*}
    G^{\k,k}(x) = -\frac{i}{4} \sum_{m\in\L} H^{(1)}_0\Big( k|x-m| \Big) e^{i m \cdot \k},
\end{align}
where $H^{(1)}_0$ denotes the Hankel function of the first kind of order zero and has the following asymptotic expansion in terms of a small argument:
\begin{align}\label{H*}
    H^{(1)}_0(s) = \frac{2i}{\pi} \sum_{m=0}^{\infty} (-1)^m \frac{s^{2m}}{2^{2m}(m!)^2} \left( \log(\hat{\g}s) - \sum_{j=1}^m \frac{1}{j} \right).
\end{align}

\subsubsection{Integral operators}

We define the integral operators $K^{\k,(-1)}_{D_i}: L^2(D_i) \to L^2(D_i)$ and $K^{\k,(0)}_{D_i}: L^2(D_i) \to 
 L^2(D_i)$ by
\begin{align*}
    &K^{\k,(-1)}_{D_i}[u](x) := -\frac{1}{2\pi} \log(\hat{\g} \d k_0) \int_{D_i} \sum_{m\in\L} e^{im\cdot\k} u(y) \upd y \Big|_{D_i}, \\
    &K^{\k,(0)}_{D_i}[u](x) := -\frac{1}{2\pi} \int_{D_i} \sum_{m\in\L} \log\Big(|x-y-m|\Big) e^{im\cdot\k} u(y) \upd y \Big|_{D_i},
\end{align*}
and the integral operators $R^{\k,(-1)}_{D_i D_j}: L^2(D_i) \to L^2(D_j)$ and $R^{\k,(0)}_{D_i D_j}: L^2(D_i) \to L^2(D_j)$ by
\begin{align*}
    &R^{\k,(-1)}_{D_i D_j}[u](x) := -\frac{1}{2\pi} \log(\hat{\g} \d k_0) \int_{D_i} \sum_{m\in\L} e^{im\cdot\k} u(y) \upd y \Big|_{D_j}, \\
    &R^{\k,(0)}_{D_i D_j}[u](x) := -\frac{1}{2\pi} \int_{D_i} \sum_{m\in\L} \log\Big(|x-y-m|\Big) e^{im\cdot\k} u(y) \upd y \Big|_{D_j},
\end{align*}
for $i=1,\dots,N$. We will provide some results which will help us in the analysis of the problem. 
\begin{definition}
We define the integral operators $M^{\d k_0}_{D_i}$ and $N^{\d k_0}_{D_i D_j}$ for $i,j=1,2$ by
\begin{align}\label{def M and N*}
    M^{\d k_0}_{D_i} := K^{\k,(-1)}_{D_i} + K^{\k,(0)}_{D_i} \ \ \ \text{ and } \ \ \ 
    N^{\d k_0}_{D_i D_j} := R^{\k,(-1)}_{D_i D_j} + R^{\k,(0)}_{D_i D_j}.
\end{align}
\end{definition}
From the asymptotic expansion of the Hankel function in \eqref{H*}, the following holds.
\begin{proposition}
For the integral operators $K^{\k,\d k_0}_{D_i}$ and $R^{\k,\d k_0}_{D_i D_j}$, defined in \eqref{K_Di} and \eqref{R_DiDj} respectively, we can write
\begin{align} \label{M and N*}
    K^{\k,\d k_0}_{D_i} = M^{\k,\d k_0}_{D_i} + O\Big(\d^2 \log(\d)\Big) \ \ \
    \text{ and } \ \ \ 
    R^{\k,\d k_0}_{D_i D_j} = N^{\k,\d k_0}_{D_i D_j}+ O\Big(\d^2 \log(\d)\Big),
\end{align}
as $\delta\to0$ and with $k_0$ fixed.
\end{proposition}

\subsubsection{Spectral results}

We have the following spectral results for the operators $K^{\k,(-1)}_{D_i}$ and $K^{\k,\d k_0}_{D_i}$, for $i=1,\dots,N$.

\begin{lemma}\label{Psi lemma*}
    Let $\n^{(\k)}_{-1,i}$ and $\Psi_{-1,i}^{(\k)}$ denote a non-zero eigenvalue and the associated eigenvector of the operator $K^{\k,(-1)}_{D_i}$, for $i=1,\dots,N$. Then, 
    \begin{align}\label{Psi_-1*}
        \n^{(\k)}_{-1,i} = - \frac{|D_i|}{2\pi}\sum_{m\in\L} e^{im\cdot\k} \ \ \ \text{ and } \ \ \ \Psi_{-1,i}^{(\k)} = \hat{\mathbbm{1}}_{D_i},
    \end{align}
    for $i=1,\dots,N$, where $\hat{\mathbbm{1}}_{D_i} := \frac{\mathbbm{1}_{D_i}}{\sqrt{|D_i|}}$ and $|D_i|$ denotes the volume of $D_i$.
\end{lemma}
\begin{proof}
    From the definition of $K^{\k,(-1)}_{D_i}$, for $i=1,\dots,N$, we observe that $K^{\k,(-1)}_{D_i}$ is independent of $x\in D_i$, and so, normalizing on $L^2(D_i)$, we get
    \begin{align*}
        \Psi_{-1,i}^{(\k)} = \hat{\mathbbm{1}}_{D_i}.
    \end{align*}
    Then, the following must hold:
    \begin{align*}
        \begin{aligned}
            \n_{-1,i}^{(\k)} \hat{\mathbbm{1}}_{D_i} = K^{\k,(-1)}_{D_i} [\hat{\mathbbm{1}}_{D_i}] &\Rightarrow \n_{-1,i}^{(\k)} \hat{\mathbbm{1}}_{D_i} = - \frac{|D_i|}{2\pi}\hat{\mathbbm{1}}_{D_i} \sum_{m\in\L} e^{im\cdot\k} \\
            &\Rightarrow \n_{-1,i}^{(\k)} = - \frac{|D_i|}{2\pi}\sum_{m\in\L} e^{im\cdot\k}.
        \end{aligned}
    \end{align*}
    This concludes the proof.
\end{proof}

\begin{lemma}\label{l expansion*}
    Let $\n^{(\k)}_{i}$ denote a non-zero eigenvalue of the operator $M^{\k,\d k_0}_{D_i}$, for $i=1,\dots,N$, in dimension 2. Then, for small $\delta$, it is approximately given by:
    \begin{align}\label{l_e approx d=2*}
        \n^{(\k)}_{i} = \log(\d k_0 \hat{\g})\n^{(\k)}_{-1,i} + \langle K^{\k,(0)}_{D_i} \Psi^{(\k)}_{-1,i}, \Psi^{(\k)}_{-1,i} \rangle + O(\d^2\log(\d)),
    \end{align}
    where $\n^{(\k)}_{-1,i}$ and $\Psi^{(\k)}_{-1,i}$ denote an eigenvalue and the associated eigenvector of the potential $K^{\k,(-1)}_{D_i}$, for $i=1,\dots,N$, respectively.
\end{lemma}
\begin{proof}
    This was proved in Lemma~2.6 of \cite{AD}.
\end{proof}

Since we have considered identical resonators, the symmetry of the system leads to the following simple result.

\begin{lemma}\label{eqaul n}
     Let $\n^{(\k)}_{-1,i}$ denote a non-zero eigenvalue of the operator $K^{\k,(-1))}_{D_i}$, for $i=1,\dots,N$. Then, it holds that
     \begin{align*}
         \n^{(\k)}_{-1,1} = \n^{(\k)}_{-1,2} = \dots = \n^{(\k)}_{-1,N} =: \n^{(\k)}_{-1}.
     \end{align*}
\end{lemma}

\subsubsection{Resonant frequencies}

We will now state a more explicit version of Theorem \ref{determinant}. This is the main result of our analysis of the two-dimensional system, which fully characterises the resonant frequencies of the periodic system.

\begin{proposition}
    The resonance problem, as $\d\to0$ and $\r\to0$, with $\d=O(\r)$ and $\r=O(\d)$, \eqref{Resonance pb} in dimensions $d=2,3$, becomes finding $\w\in\mathbb{C}$ such that
    \begin{align*}
        \det\Big(\mathcal{K}^{\k}(\w)\Big) = 0,
    \end{align*}
    where
    \begin{align}\label{K*}
        \mathcal{K}^{\k}(\w)_{ij} = 
            \begin{cases}
                \langle N^{\k,\d k_0}_{D_i D_{i+1 \lfloor N \rfloor }} \hat{1}_{D_i}, \hat{1}_{D_{(i+1\lfloor N \rfloor)}} \rangle , &\text{ if $i=j$,} \\
                - \mathscr{A}^{\k}_{i}(\w,\d) \langle N^{\k,\d k_0}_{D_j D_i} \hat{1}_{D_j}, \hat{1}_{D_i} \rangle \langle N^{\k,\d k_0}_{D_i D_{i+1 \lfloor N \rfloor }} \hat{1}_{D_i}, \hat{1}_{D_{(i+1\lfloor N \rfloor)}} \rangle, &\text{ if $i \ne j$.}
            \end{cases}
    \end{align}
    Here, $k_0=\omega\sqrt{\mu_0\varepsilon_0}$ and
    \begin{align}\label{A_i*}
        \mathscr{A}^{\k}_{i}(\w,\d) := \frac{\d^2 \w^2 \xi(\w)}{1 - \d^2 \w^2 \xi(\w) \n^{(\k)}}, \quad i=1,...,N.
    \end{align}
    with $\n^{(\k)}$ denoting a non-zero eigenvalue of the potential $M^{\k,\d k_0}_{D_i}$, for $i=1,2,\dots,N$.
\end{proposition}

\begin{proof}
    This is a direct consequence of Theorem \ref{determinant}. We just have to apply \eqref{M and N*} and \eqref{Psi_-1*} to  \eqref{K} and get
    \begin{align*}
        \mathcal{K}^{\k}(\w)_{ij} = 
            \begin{cases}
                \langle N^{\k,\d k_0}_{D_i D_{i+1 \lfloor N \rfloor }} \hat{1}_{D_i}, \hat{1}_{D_{(i+1\lfloor N \rfloor)}} \rangle , &\text{ if $i=j$,} \\
                - \mathscr{A}^{\k}_{i}(\w,\d) \langle N^{\k,\d k_0}_{D_j D_i} \hat{1}_{D_j}, \hat{1}_{D_i} \rangle \langle N^{\k,\d k_0}_{D_i D_{i+1 \lfloor N \rfloor }} \hat{1}_{D_i}, \hat{1}_{D_{(i+1\lfloor N \rfloor)}} \rangle, &\text{ if $i \ne j$,}
            \end{cases}
    \end{align*}
    which gives the desired result.
\end{proof}

\section{Conclusion}

We have used analytic methods to understand the dispersive nature of photonic crystals fabricated from metals with singular permittivities. In particular, we considered a Drude--Lorentz model inspired by halide perovskites that has poles in the lower complex plane. For a one-dimensional system, we characterised the effect that each feature of this model has on the dispersion relation. In particular, we showed that the introduction of singularities leads to the creation of countably many band gaps near the poles, whereas the introduction of damping smooths out the band gap structure. Finally, we showed how the integral methods developed in \cite{AD,alexopoulos2022mathematical} can be used to extend this theory to multi-dimensional systems.


\section*{Conflicts of interest}

The authors have no conflicts of interest to disclose.

\section*{Data availability}

The code used to perform the numerical simulations presented in this work can be found at \url{https://doi.org/10.5281/zenodo.8055547}. No other data were generated in this project.

\section*{Acknowledgements}

The work of KA was supported by ETH Z\"urich under the project ETH-34 20-2. The work of BD was funded by the Engineering and Physical Sciences Research Council through a fellowship with grant number EP/X027422/1.

\bibliographystyle{abbrv}
\bibliography{references}{}

\end{document}